\documentclass[11pt]{amsart}
\usepackage[T2A,T1]{fontenc}
\usepackage{lmodern}
\usepackage{amsmath, amsthm, amssymb, amsfonts}
\usepackage[normalem]{ulem}
\usepackage{mathrsfs}
\usepackage{verbatim} 
\usepackage{longtable}
\usepackage{mathtools}

\usepackage{tikz}
\usetikzlibrary{decorations.pathmorphing}
\tikzset{snake it/.style={decorate, decoration=snake}}
\usepackage{caption}
\usepackage{tikz-cd}
\usetikzlibrary{arrows}
\usepackage{hyperref}

\theoremstyle{plain}
\newtheorem{thm}{Theorem}[section]
\newtheorem{cor}[thm]{Corollary}
\newtheorem{lem}[thm]{Lemma}
\newtheorem{prop}[thm]{Proposition}
\newtheorem{conj}[thm]{Conjecture}
\newtheorem{question}[thm]{Question}

\theoremstyle{definition}
\newtheorem{defn}[thm]{Definition}

\theoremstyle{remark}
\newtheorem{rmk}[thm]{Remark}

\newcommand{\BC}{{\mathbb{C}}}

\newcommand{\BH}{{\mathbb{H}}}

\newcommand{\BP}{{\mathbb{P}}}
\newcommand{\BQ}{{\mathbb{Q}}}

\newcommand{\BZ}{{\mathbb{Z}}}

\newcommand{\CH}{{\mathcal H}}

\newcommand{\CJ}{{\mathcal J}}
\newcommand{\CK}{{\mathcal K}}
\newcommand{\CL}{{\mathcal L}}
\newcommand{\CM}{{\mathcal M}}

\newcommand{\CO}{{\mathcal O}}
\newcommand{\CP}{{\mathcal P}}

\newcommand{\FM}{{\mathfrak{M}}}

\newcommand{\Fp}{{\mathfrak{p}}}
\newcommand{\Fq}{{\mathfrak{q}}}

\newcommand{\FC}{{\mathfrak{C}}}
\newcommand{\FF}{{\mathfrak{F}}}

\newcommand{\td}{{\mathrm{td}}}

\DeclareFontFamily{OT1}{rsfs}{}
\DeclareFontShape{OT1}{rsfs}{n}{it}{<-> rsfs10}{}
\DeclareMathAlphabet{\curly}{OT1}{rsfs}{n}{it}

\newcommand\Hom{\operatorname{Hom}}

\newcommand{\Mbar}{\overline{\mathcal{M}}}

\newcommand{\Chow}{\mathrm{CH}}

\newcommand{\Coh}{\mathrm{Coh}}
\newcommand{\Gr}{\mathrm{Gr}}

\newcommand{\Mgint}{\overline{\mathcal{M}}_g^{\mathrm{int}}}

\newcommand{\Jbar}{\overline{J}}

\newcommand{\Mtl}{\overline{\mathcal{M}}^\mathrm{tl}}
\newcommand{\qs}{\mathrm{qs}}
\newcommand{\pic}{\mathfrak{Pic}}
\newcommand{\sfa}{\mathsf{a}}
\newcommand{\uniDR}{\mathsf{uniDR}}
\newcommand{\cc}{c}
\newcommand{\dr}{\mathsf{DR}}

\makeatletter
\def\easycyrsymbol#1{\mathord{\mathchoice
  {\mbox{\fontsize\tf@size\z@\usefont{T2A}{\rmdefault}{m}{n}#1}}
  {\mbox{\fontsize\tf@size\z@\usefont{T2A}{\rmdefault}{m}{n}#1}}
  {\mbox{\fontsize\sf@size\z@\usefont{T2A}{\rmdefault}{m}{n}#1}}
  {\mbox{\fontsize\ssf@size\z@\usefont{T2A}{\rmdefault}{m}{n}#1}}
}}
\makeatother
\newcommand{\RB}{\easycyrsymbol{\CYRB}}

\makeatletter
\let\@wraptoccontribs\wraptoccontribs

\begin{document}
\title[The intrinsic cohomology ring of the universal compactified Jacobian]{The intrinsic cohomology ring of the universal compactified Jacobian over the moduli space of stable curves}
\date{\today}

\author[Y. Bae]{Younghan Bae}
\address{University of Michigan}
\email{younghan@umich.edu}

\author[D. Maulik]{Davesh Maulik}
\address{Massachusetts Institute of Technology}
\email{maulik@mit.edu}

\author[J. Shen]{Junliang Shen}
\address{Yale University}
\email{junliang.shen@yale.edu}

\author[Q. Yin]{Qizheng Yin}
\address{Peking University}
\email{qizheng@math.pku.edu.cn}

\address{Stockholm University}
\email{dan.petersen@math.su.se}

\begin{abstract}
The purpose of this paper is to study the cohomology rings of universal compactified Jacobians. Over the moduli space $\overline{\CM}_{g,n}$ of Deligne--Mumford stable marked curves with~$n\geq 1$, on the one hand we show that the cohomology ring of a universal fine compactified Jacobian is sensitive to the choice of a nondegenerate stability condition which answers a question of Pandharipande; on the other hand, we prove that the cohomology ring admits a degeneration via the perverse filtration which is independent of the (nondegenerate) stability condition. The latter defines the intrinsic cohomology ring of the universal compactified Jacobian which only relies on $g,n$.

Our main tools include the support theorems, the recently developed Fourier theory for dualizable abelian fibrations, and the universal double ramification cycle relations associated with the universal Picard stack. 
\end{abstract}

\maketitle

\setcounter{tocdepth}{1} 

\tableofcontents
\setcounter{section}{-1}

\section{Introduction} \label{sec:0}

\subsection{Overview}
Throughout, we work over the complex numbers $\BC$. 

The Jacobian variety associated with a nonsingular projective curve is a fundamental geometric object. Universally, we can consider the relative Jacobian over the moduli space $\CM_{g,n}$ of nonsingular genus $g$ marked curves, and their compactifications over the Deligne--Mumford moduli space $\overline{\CM}_{g,n}$ of stable marked curves.

When working over $\overline{\CM}_{g,n}$, the universal Jacobian admits compactifications known as universal \emph{fine compactified Jacobians}; the geometry of a fine compactified Jacobian depends on the choice of a stability condition. As we will show in Theorem~\ref{thm0.6}, the cohomology ring of a fine compactified Jacobian also depends on the chosen stability condition. We will demonstrate in Theorem \ref{thm0.7} that, although these cohomology rings may differ, they admit a common degeneration. This degeneration is constructed using the \emph{perverse filtration}, and should be viewed as ``the intrinsic cohomology ring'' of the compactified Jacobian over the moduli space of stable curves, which does not rely on the stability condition.

We note that the perverse filtration has played a key role in the study of degenerations of Jacobians of curves. For example, the perverse filtration associated with the Hitchin system detects the weight filtration of the character variety via the (now proven) $P=W$ conjecture in non-abelian Hodge theory \cite{dCHM1, MS_PW, HMMS, MSY}; the perverse filtration of a compactified Jacobian calculates the Gopakumar--Vafa invariants associated with a (possibly singular) curve \cite{MT}, and is closely related to knot invariants \cite{ORS, MY, MS1}. In this paper, the natural appearance of the perverse filtration is due to its compatibility with the Fourier transform developed in \cite{MSY}.

\subsection{Jacobians}\label{sec_trivial}

Let $\CM_g$ be the moduli space of nonsingular projective irreducible curves of genus $g$. The degree $d$ universal Jacobian $J^d_g$ is smooth and proper over $\mathcal{\CM}_g$:
\[
\pi_d: J^d_g \to \CM_g, \quad \left(C_b, L \in \mathrm{Pic}^d(C_b)\right) \mapsto [C_b].
\]
Although the (birational) geometry of $J^d_g$ is sensitive to $d$ (see \cite[Theorem 1.7]{BFV}), the cohomology ring is not.

\begin{prop}\label{thm1}
    For integers $d,d'$, there is an isomorphism 
    \[
    H^*(J^d_g, \BQ) \simeq H^*(J_g^{d'}, \BQ)
    \]
    of graded $H^*(\CM_g, \BQ)$-algebras.
\end{prop}

Indeed, there is a morphism 
\[
\varphi: J_g^d \to J_g^{d(2g-2)}, \quad (C_b, L) \mapsto (C_b, L^{\otimes (2g - 2)})
\]
which is fiberwise an isogeny up to translation over $\CM_g$. The pullback induces a morphism 
\[
\varphi^*: \pi_{d(2g-2) *} \BQ \to \pi_{d *} \BQ
\]
compatible with the cup-product, and it is an isomorphism which can be verified on the stalks. We thus obtain an isomorphism 
\begin{equation}\label{iso_trivial}
\varphi^*: H^*(J^{d(2g-2)}_g, \BQ) \xrightarrow{\simeq} H^*(J^{d}_g, \BQ) 
\end{equation}
of graded $H^*(\CM_g, \BQ)$-algebras. Finally, we have the natural isomorphism $J^{d(2g-2)}_g \simeq J^0_g$ over~$\CM_g$ induced by the relative canonical bundle, so the left-hand side of (\ref{iso_trivial}) is $d$-independent.

In Appendix A, we generalize Proposition \ref{thm1} to any torsor under an abelian scheme over a nonsingular base by extending the strategy above. However, it is not clear to us how this proof can be generalized to compactified Jacobians with singular fibers.  

Instead, we will generalize another, more elaborate, proof of Proposition \ref{thm1} following the arguments of Beauville \cite{B0, B}.  We sketch the proof as follows.

The idea for the second proof is to use a derived equivalence between certain gerbes over~$J^d_g$ and $J^0_g$ respectively. By passing to cohomology, this yields a Fourier transform connecting $H^*(J^d_g, \BQ)$ with $H^*(J^0_g, \BQ)$, which relates the cup-product for $J^d_g$ to the convolution product for $J^0_g$:
\begin{equation}\label{Fourier=Convolution1}
(H^*(J^d_g, \BQ), \cup) \longleftarrow \mathrm{Fourier~ Transform} \longrightarrow (H^*(J^0_g, \BQ), \ast^d).
\end{equation}
The $d$-independence of the cup-product on $H^*(J^d_g, \BQ)$ is thus reduced to the $d$-independence of the convolution product on $H^*(J^0_g, \BQ)$, which is straightforward.

As we discuss in the next section, by applying some recent developments in the study of perverse filtrations of abelian fibrations, this argument can be generalized to handle the case with singular fibers.

\subsection{Compactified Jacobians}

We consider a partial compactification of $\CM_g$,
\[\CM_g \subset \overline{\CM}_g^{\mathrm{int}} \subset \overline{\CM}_g,
\]
where $\overline{\CM}_g^{\mathrm{int}}$ is the moduli space of \emph{integral} stable curves. For any degree $d$, the universal Jacobian $\pi_d: J^d_g \to \CM_g$ admits a natural extension 
\[
J^d_{g} \subset \overline{J}^{\mathrm{int},d}_g. 
\]
Here $\overline{J}^{\mathrm{int},d}_g$ is the degree $d$ universal \emph{compactified Jacobian} parameterizing pairs $(C_b, F)$ with \mbox{$[C_b] \in \overline{\CM}^{\mathrm{int}}_g$} an integral stable curve and $F$ a (generically) rank $1$ torsion-free sheaf on~$C_b$ satisfying
\[
\chi(F) = d+1-g.
\]
The universal compactified Jacobian admits a natural proper morphism
\begin{equation*}\label{pi_d_int}
\pi_d: \overline{J}^{\mathrm{int},d}_g \to \overline{\CM}_g^{\mathrm{int}} 
 , \quad (C_b, F) \mapsto [C_b]
\end{equation*}
extending the smooth morphism $\pi_d: J^d_g \to \CM_g$. For a singular nodal curve 
\[
[C_b] \in \overline{\CM}_g^{\mathrm{int}} \setminus \CM_g,
\]
the fiber $\pi_d^{-1}([C_b])$ is the compactified Jacobian associated with the curve $C_b$, which is irreducible and contains the Jacobian of $C_b$ (parameterizing line bundles) as a Zariski dense open subset.

The map $\pi_d: \overline{J}^{\mathrm{int},d}_g \to \overline{\CM}_g^{\mathrm{int}}$ endows the cohomology $H^*( \overline{J}^{\mathrm{int},d}_g, \BQ)$ with an extra structure --- the perverse filtration \cite{BBD}:
\begin{equation}\label{p-filtration}
P_0 H^*(\overline{J}^{\mathrm{int},d}_g, \BQ) \subset P_1H^*(\overline{J}^{\mathrm{int},d}_g, \BQ) \subset \cdots \subset P_{2g}H^*(\overline{J}^{\mathrm{int},d}_g, \BQ) = H^*(\overline{J}_g^{\mathrm{int},d}, \BQ).
\end{equation}
The following theorem was proven in \cite{MSY}, which shows that the degeneration of the cohomology $H^*(\overline{J}^{\mathrm{int},d}_g, \BQ)$ via the perverse filtration is a bigraded $H^*(\overline{\CM}_g^{\mathrm{int}}, \BQ)$-algebra.

\begin{thm}[\cite{MSY}]
The perverse filtration \eqref{p-filtration} is multiplicative with respect to the cup-product, \emph{i.e.}, for integers $k,l$ we have
\[
\cup: P_k H^*(\overline{J}^{\mathrm{int},d}_g, \BQ) \otimes P_l H^*(\overline{J}^{\mathrm{int},d}_g, \BQ) \to  P_{k + l} H^*(\overline{J}^{\mathrm{int},d}_g, \BQ).
\]
In particular, the associated graded
\[
\BH^{\mathrm{int},d}_g:= \bigoplus_{k,m} \mathrm{Gr}^P_kH^m(\overline{J}^{\mathrm{int},d}_g, \BQ)
\]
is a bigraded $H^*(\overline{\CM}^{\mathrm{int}}_g, \BQ)$-algebra induced by the cup-product.
\end{thm}

By definition, any element in $H^m(\overline{\CM}^{\mathrm{int}}_g, \BQ)$ has bigrading $(0,m)$. Clearly, $\BH^{\mathrm{int},d}_g$ is isomorphic to $H^*(\overline{J}^{\mathrm{int},d}_g, \BQ)$ as $\BQ$-vector spaces, but the ring structures may be different. The following result shows the $d$-independence of the algebra $\BH^{\mathrm{int},d}_g$.

\begin{thm}\label{thm0.3}
For integers $d,d'$, there is an isomorphism
\[
\BH^{\mathrm{int},d}_g \simeq  \BH^{\mathrm{int},d'}_g 
\]
of bigraded $H^*(\overline{\CM}^{\mathrm{int}}_g, \BQ)$-algebras.
\end{thm}

We in fact prove a stronger statement concerning a family of integral locally planar curves; see Theorem \ref{thm:mainint} and Corollary \ref{cor:mainint}. This recovers Theorem \ref{thm0.3} immediately since nodal singularities are planar. 

We can also consider the universal compactified Jacobian over the locus of integral curves in~$\overline{\CM}_{g,n}$; there are obvious isomorphisms between the compactified Jacobians of various degrees as long as $n\geq 1$.

\begin{rmk}
It is natural to ask if there is an isomorphism  
\[
H^*(\overline{J}^{\mathrm{int},d}_g, \BQ) \simeq  H^*(\overline{J}^{\mathrm{int},d'}_g, \BQ) 
\]
of graded $H^*(\overline{\CM}^{\mathrm{int}}_g, \BQ)$-algebras without passing to the associated graded. In view of Theorem~\ref{thm0.6} below, we expect that the answer is negative.
\end{rmk}

\subsection{Fine compactified Jacobians}

Next, we consider the moduli space $\overline{\CM}_{g,n}$ of stable curves with $n$ markings. In order to further extend the compactified Jacobians to the locus of reducible curves, \emph{stability conditions} are needed. In this paper we work with the stability conditions introduced by Kass--Pagani \cite{KP}, and independently by Melo \cite{Melo}.

Roughly, a stability condition in the sense of \cite{KP} is an assignment of a rational number to every irreducible component of every stable marked curve
\[
(C_b, x_1,\cdots, x_n) \in \overline{\CM}_{g,n},
\]
satisfying certain compatibility conditions.
When we fix the degree $d$, the space of stability conditions has a wall-and-chamber structure. The universal \emph{fine compactified Jacobian} for a given stability condition is the moduli space parameterizing rank $1$ torsion-free sheaves on stable marked curves which satisfy the stability inequality. We denote it by~$\overline{J}_{g,n}^{d,\phi}$ to indicate its dependence on the genus $g$, the number of markings $n$, the degree of the torsion-free sheaves~$d$, and the stability condition $\phi$. For a \emph{nondegenerate} stability condition $\phi$, \emph{i.e.}, there are no strictly semistable sheaves, the universal fine compactified Jacobian $\overline{J}_{g,n}^{d,\phi}$ is a nonsingular proper Deligne--Mumford stack which contains the universal compactified Jacobian of integral curves as a Zariski dense open subset.

The following shows that the cohomology ring of $\overline{J}_{g,n}^{d,\phi}$ depends on the stability condition $\phi$; this answers a question of Pandharipande \cite[Question A]{question}.

\begin{thm}\label{thm0.6}
For $g \geq 4$ and any integers $d, d'$, there exist nondegenerate stability conditions~$\phi, \phi'$ of degrees $d,d'$ respectively such that
\[
H^*(\overline{J}_{g,1}^{d,\phi}, \BQ) \not\simeq H^*(\overline{J}_{g,1}^{d',\phi'}, \BQ)
\]
as graded $H^*(\overline{\CM}_{g,1}, \BQ)$-algebras.
\end{thm}

On the other hand, we may consider the perverse filtration $P_\bullet H^*(\overline{J}_{g,n}^{d,\phi}, \BQ)$ associated with the natural proper morphism
\[
\pi_d: \overline{J}_{g,n}^{d,\phi} \to \overline{\CM}_{g,n}
\]
and the associated graded
\[
\BH^{d,\phi}_{g,n}:= \bigoplus_{k,m} \mathrm{Gr}_k^PH^m(\overline{J}^{d,\phi}_{g,n}, \BQ).
\]

\begin{thm}\label{thm0.7}
Assume that $\phi, \phi'$ are nondegenerate stability conditions of degrees $d,d'$ respectively, and that $n \geq 1$. We have the following.
\begin{enumerate}
\item[(i)] The associated graded $\BH^{d, \phi}_{g,n}$ is a bigraded algebra over $H^*(\overline{\CM}_{g,n}, \BQ)$ with elements in $H^m(\overline{\CM}_{g,n}, \BQ)$ of bigrading $(0,m)$.
       
\item[(ii)] There is an isomorphism 
\[
\BH^{d,\phi}_{g,n} \simeq \BH^{d',\phi'}_{g,n}
\]
of bigraded $H^*(\overline{\CM}_{g,n}, \BQ)$-algebras.
\end{enumerate}
\end{thm}

Theorems \ref{thm0.6} and \ref{thm0.7} suggest that the induced cup-product on the associated graded with respect to the perverse filtration is ``intrinsic'' to the moduli space of stable marked curves.

\begin{rmk}
When $n=0$, nondegenerate stability conditions of degree $d$ exist if and only if 
\begin{equation}\label{existence}
\mathrm{gcd}(d + 1 - g, 2g - 2) = 1;
\end{equation}
see \cite[Remark 5.12]{KP}. Thus, the best we can expect is that both statements of Theorem \ref{thm0.7} hold for those values of $d$ satisfying the numerical condition \eqref{existence}. Our proof does not cover this case due to the absence of a nondegenerate stability condition of degree $0$.
\end{rmk}

We note that the proof of Theorem \ref{thm0.7} also yields a parallel result for fine compactified Jacobians of a reduced locally planar curve which may be of independent interest. To state the result, now we assume that $C_0$ is a reduced locally planar curve of arithmetic genus $g$. For any choice of a nondegenerate stability condition $\phi$, \emph{i.e.}, there are no strictly semistable sheaves, Migliorini--Shende--Viviani \cite{MSV} introduced a perverse filtration on the cohomology of the corresponding fine compactified Jacobian
\begin{equation}\label{perv_fil_C}
P_0H^*(\overline{J}^\phi_{C_0}, \BQ) \subset P_1 H^*(\overline{J}^\phi_{C_0}, \BQ) \subset \cdots \subset P_{2g}H^*(\overline{J}^\phi_{C_0}, \BQ) = H^*(\overline{J}^\phi_{C_0}, \BQ).
\end{equation}
Similarly, we consider the associated graded
\[
\BH^\phi_{C_0}:= \bigoplus_{k,m} \mathrm{Gr}_k^PH^m(\overline{J}_{C_0}^\phi, \BQ).
\]

\begin{thm}\label{thm0.8}
Assume that $\phi, \phi'$ are nondegenerate stability conditions for the reduced locally planar curve $C_0$. We have the following.
\begin{enumerate}
\item[(i)] The associated graded $\BH^{\phi}_{C_0}$ is a bigraded $\BQ$-algebra.
       
\item[(ii)] There is an isomorphism 
\[
\BH^{\phi}_{C_0} \simeq \BH^{\phi'}_{C_0}
\]
of bigraded $\BQ$-algebras.
\end{enumerate}
\end{thm}

Theorems \ref{thm0.7} and \ref{thm0.8} are deduced from a more general sheaf-theoretic statement for relative fine compactified Jacobians associated with locally planar reduced curves; see Theorem \ref{thm_main_reduced}.


\subsection{Ideas of the proofs}
The key idea in the proofs of Theorems \ref{thm0.3}, \ref{thm0.7}, and \ref{thm0.8} is the observation that the Fourier transform controls the cup-product, generalizing the proof sketched at the end of Section \ref{sec_trivial}.

To carry out this idea when there are singular curves, we apply the Fourier transform induced by the sheaves constructed by Arinkin \cite{A2} and Melo--Rapagnetta--Viviani \cite{MRV2} for the Fourier--Mukai duality. By the recent work of Maulik--Shen--Yin \cite{MSY}, the Fourier transform interacts naturally with the (multiplicative) perverse filtration. Furthermore, we show in this paper that the Fourier transform and the induced cup-product on the associated graded (with respect to the perverse filtration) provides an analogue of (\ref{Fourier=Convolution1}) when there are singular curves. For example, in the setting of Theorem \ref{thm0.7} we have
\begin{equation}\label{Fourier=Convolution2}
\left(\mathrm{Gr}^P_\bullet H^*(\Jbar^{d, \phi}_{g,n}, \BQ), \overline{\cup}  \right) \longleftarrow \mathrm{Fourier~ Transform} \longrightarrow
\left(\mathrm{Gr}^P_\bullet H^*(\Jbar^{0, \phi_0}_{g,n}, \BQ), \overline{\ast}_{\mathrm{red}}^\phi \right)
\end{equation}
where $\phi_0$ is a fixed nondegenerate stability condition and $\overline{\ast}_{\mathrm{red}}^\phi$ is the reduced convolution product --- the natural convolution product on the associated graded. We complete the proof of Theorem \ref{thm0.7} by showing that the reduced convolution product on the right-hand side of~(\ref{Fourier=Convolution2}), which \emph{a priori} relies on $\phi$, is in fact independent of $d, \phi$.

From the perspective of the Fourier transform, the induced cup-product on the associated graded with respect to the perverse filtration is a more natural ring structure on the cohomology group of the universal compactified Jacobian. We view this as the \emph{intrinsic cohomology ring} of the universal compactified Jacobian; this ring is defined via a choice of compactification, but is eventually independent of the compactification. It is interesting to explore if the intrinsic cohomology ring can be realized \emph{geometrically}, \emph{i.e.}, as (part of) the cohomology ring of some space related to the universal Jacobian.

On the other hand, we prove in Theorem \ref{thm0.6} that the actual ring structure of $H^*(\Jbar^{d, \phi}_{g,n}, \BQ)$ is very sensitive to the choice of nondegenerate stability condition $\phi$; this can already be seen by considering monomials of divisors on fine compactified Jacobians.

\subsection{Relations to other work and further discussions} 

We conclude the introduction by discussing some relations to other work.

\subsubsection{Compactified Jacobians} There are other versions of (relative) compactified Jacobians, which generalize the ones given by the polarization stability condition considered in this paper; we refer to \cite{PaTo, FPV} and the references therein for more details. We expect that our theory via the Fourier--Mukai duality and the support theorem can also be applied to these versions.

\subsubsection{Enhanced $\chi$-independence}
The relative Jacobian (of a given degree) of the moduli space of nonsingular degree $d$ planar curves admits a natural compactification by Le Potier \cite{LP}. The Le Potier moduli space $M_{d,\chi}$ parameterizes semistable $1$-dimensional sheaves on $\BP^2$ with Fitting support of degree $d$ and Euler characteristic $\chi$. We assume $(d,\chi)=1$ so that the moduli spaces are nonsingular; a story parallel to the case of stable curves is expected to hold for this compactification:
\begin{enumerate}
    \item[(i)] The cohomology group $H^*(M_{d,\chi}, \BQ)$ is known to be $\chi$-independent \cite{Bou, MS}.  
    \item[(ii)] The ring structure of $H^*(M_{d,\chi}, \BQ)$ is highly sensitive to $\chi$ \cite{LMP}.
    \item[(iii)] It was conjectured by Kononov--Moreira--Lim--Pi that the associated graded 
    \[
    \mathrm{Gr}^P_\bullet H^*(M_{d,\chi}, \BQ)
    \]
with respect to the perverse filtration (defined by the support map $M_{d,\chi} \to |\CO_{\BP^2}(d)|$) is naturally a $\BQ$-algebra induced by the cup-product, which is further $\chi$-independent; this was confirmed for $d=5$ \cite{KLMP}. 
    \end{enumerate}
Since the Le Potier compactification of the relative Jacobian of planar curves contains fibers which are not given by fine compactified Jacobians of reduced curves, our theory (which relies on the Fourier--Mukai duality) cannot be applied to prove the conjecture (iii). On the other hand, the recently discovered $P=C$ phenomenon for $M_{d,\chi}$ \cite{KPS, KLMP} suggests that a version of the Arinkin sheaf may be extended over the non-reduced planar curves.

\subsubsection{Lagrangian fibrations}
In general, the induced cup-product — assuming it is well defined — on the associated graded with respect to the perverse filtration does not yield the same ring as the ordinary cohomology ring (\emph{e.g.}~one can compare Theorem \ref{thm0.6} with Theorem \ref{thm0.7}). On the other hand, it was conjectured (\emph{c.f.}~\cite[Section 3]{BMSY}) that the perverse filtration associated with a \emph{Lagrangian fibration} $f:M\to B$ admits a multiplicative splitting (even \emph{motivically}); in particular, one expects an isomorphism
\begin{equation}\label{equiv}
\left(\mathrm{Gr}^P_\bullet H^*(M, \BQ), \overline{\cup} \right) \simeq \left(H^*(M, \BQ), \cup \right).
\end{equation}
For Lagrangian fibrations associated with compact hyper-K\"ahler varieties, this was proven in~\cite{SY} using the Hodge decomposition and the Looijenga--Lunts--Verbitsky Lie algebra; for the Hitchin system associated with $\mathrm{GL}_n$, this is a consequence of the $P=W$ conjecture~\cite{dCHM1, MS_PW, HMMS, MSY}.

The isomorphism (\ref{equiv}) is further expected to hold sheaf-theoretically over $B$, and the discussion above suggests that its obstruction is global. Hence, we expect that the ``local'' perverse filtration (\ref{perv_fil_C}) admits a multiplicative splitting. This, combined with Theorem \ref{thm0.8}, leads us to conjecture the following.

\begin{conj}
    Under the assumption of Theorem \ref{thm0.8}, for the nondegenerate stability conditions~$\phi,\phi'$, there is an isomorphism
    \[
    H^*(\Jbar^\phi_{C_0}, \BQ) \simeq H^*(\Jbar^{\phi'}_{C_0}, \BQ)
    \]
    of graded $\BQ$-algebras.
\end{conj}
\subsubsection{Intrinsic tautological ring}
For fine compactified Jacobians $\Jbar_{g,n}^{d,\phi}$ over the moduli space of stable curves, the universal family induces a natural notion of tautological subring
\[
RH^*(\Jbar_{g,n}^{d,\phi}) \subset H^*(\Jbar_{g,n}^{d,\phi},\BQ);
\]
we refer to Section~\ref{sec:4.2} and the references therein for details.

In analogy with the intrinsic cohomology ring, it is natural to take the associated graded of the tautological ring with respect to the perverse filtration
\begin{equation}\label{eq:RH}
    \mathbb{RH}^{d,\phi}_{g,n} := \bigoplus_{k,m} \mathrm{Gr}_k^P RH^m(\overline{J}^{d,\phi}_{g,n}) \subset \mathbb{H}^{d,\phi}_{g,n}.
\end{equation}
\begin{question}
    Does the isomorphism in Theorem \ref{thm0.7}(ii) respect the subring \eqref{eq:RH}? If so, can we describe the structure of $\mathbb{RH}^{d,\phi}_{g,n}$?
\end{question}

\subsection{Acknowledgements}
We are grateful to Dima Arinkin, Yang Cao, Soumik Ghosh, Sam Molcho, Miguel Moreira, Rahul Pandharipande, Dan Petersen, Weite Pi, Aaron Pixton, Johannes Schmitt, and Charles Vial for helpful discussions. The results proven here were first presented by J.S.~and Q.Y.~in the workshop ``A Saturday of Moduli'' at ETH Z\"urich in June 2025.

Y.B.~was supported by June E Huh Visiting Fellowship. D.M.~was supported by a Simons Investigator Grant. J.S.~was supported by the NSF Grant DMS-2301474 and a Sloan Research Fellowship.

\section{Compactified Jacobians} \label{sec:2}
In this section, we consider integral locally planar curves and their compactified Jacobians. We begin by reviewing results in \cite{MSY} which provide a new characterization of the perverse filtration on the cohomology of a compactified Jacobian family and establish its multiplicativity. We then prove that the associated graded with respect to the perverse filtration has a ring structure which is independent of the degree of the compactified Jacobian; see Theorem~\ref{thm:mainint} and Corollary \ref{cor:mainint}.

\subsection{Perverse filtration} \label{sec:2.1}

Let $C \to B$ be a flat projective family of integral locally planar curves of arithmetic genus $g$ over a nonsingular quasiprojective base variety $B$, and for any integer~$d$, let
\[
\pi_d: \Jbar^d_C \to B
\]
be the associated compactifed Jacobian family parameterizing degree $d$ rank $1$ torsion-free sheaves. We assume that the total space $\Jbar^d_C$ is nonsingular. Of particular interest is the case~$B = \Mgint$, \emph{i.e.}, the moduli stack of integral stable curves of genus $g$. Using the level structures constructed in \cite{ACV}, we can always reduce this case to when the base $B$ is a quasiprojective variety.

The Beilinson--Bernstein--Deligne--Gabber decomposition theorem \cite{BBD}, applied to the proper map $\pi_d: \Jbar^d_C \to B$ and the constant sheaf $\BQ_{\Jbar_C^d}$, yields a \emph{non-canonical} decomposition
\begin{equation} \label{eq:BBD}
\pi_{d*}\BQ_{\Jbar_C^d} \simeq \bigoplus_{k = 0}^{2g}\CH^d_{(k)}, \quad \CH^d_{(k)} := {^\Fp}H^{k + \dim B}(\pi_{d*}\BQ_{\Jbar_C^d})[-k - \dim B] \in D^b_c(B).
\end{equation}
Here ${^\Fp}H^i(-)$ is the $i$-th perverse cohomology functor. What remain canonical are the perverse truncation functors ${^\Fp}\tau_{\leq \bullet}(-), {^\Fp}\tau_{\geq \bullet}(-)$ applied to the pushforward complex $\pi_{d*}\BQ_{\Jbar_C^d}$: for \mbox{$0 \leq k \leq 2g$}, there are natural distinguished triangles
\begin{gather}
{^\Fp}\tau_{\leq k + \dim B}(\pi_{d*}\BQ_{\Jbar_C^d}) \to \pi_{d*}\BQ_{\Jbar_C^d} \to {^\Fp}\tau_{\geq k + 1+ \dim B}(\pi_{d*}\BQ_{\Jbar_C^d}) \xrightarrow{+1}, \nonumber\\
{^\Fp}\tau_{\leq k - 1+ \dim B}(\pi_{d*}\BQ_{\Jbar_C^d}) \to {^\Fp}\tau_{\leq k + \dim B}(\pi_{d*}\BQ_{\Jbar_C^d}) \to \CH^d_{(k)} \xrightarrow{+1}. \label{eq:disttri}
\end{gather}

Taking global cohomology and setting
\[
P_kH^*(\Jbar^d_C, \BQ) := \mathrm{Im}\left(H^*(B, {^\Fp}\tau_{\leq k + \dim B}(\pi_{d*}\BQ_{\Jbar_C^d})) \to H^*(B, \pi_{d*}\BQ_{\Jbar_C^d}) = H^*(\Jbar_C^d, \BQ)\right),
\]
we obtain the perverse filtration
\[
P_0 H^*(\overline{J}^d_C, \BQ) \subset P_1H^*(\overline{J}^d_C, \BQ) \subset \cdots \subset P_{2g}H^*(\overline{J}^d_C, \BQ) = H^*(\overline{J}^d_C, \BQ)
\]
and the canonical identification
\[
\Gr_k^PH^*(\Jbar_C^d, \BQ) = H^*(B, \CH^d_{(k)}).
\]

\subsection{Perverse filtration via Fourier} \label{sec:2.2}

A new characterization of the perverse filtration $P_\bullet H^*(\Jbar_C^d, \BQ)$ (as well as its sheaf-theoretic counterpart) was given in \cite{MSY} using Fourier transforms. We review the construction and key results needed for the main Theorem \ref{thm:mainint}. These results will be axiomatized and extended to more general situations in Section \ref{sec:3}.

We fix once and for all
\[
\Jbar^\vee_C:= \Jbar^0_C, \quad \pi^\vee:= \pi_0: \Jbar^\vee_C \to B.
\]
The perverse truncations ${^\Fp}\tau_{\leq \bullet}(\pi^\vee_*\BQ_{\Jbar^\vee_C}), {^\Fp}\tau_{\geq \bullet}(\pi^\vee_*\BQ_{\Jbar^\vee_C})$, the shifted perverse sheaves $\CH^\vee_{(\bullet)}$, and the perverse filtration $P_\bullet H^*(\Jbar^\vee_C, \BQ)$ are defined accordingly.

By considering families of rank $1$ torsion-free sheaves trivialized along an $r$-fold multisection
\[
D \subset C \to B,
\]
whose existence follows from \cite[Theorem 1.1]{Cle}, and by applying Arinkin's construction~\cite{A2}, we obtain a Poincar\'e sheaf
\[
\CP^d \in \Coh(\overline{\mathcal{J}}^\vee_C \times \overline{\mathcal{J}}^d_C)_{(d, 0)}
\]
on the relative product of certain $\mu_r$-gerbes $\overline{\mathcal{J}}^\vee_C, \overline{\mathcal{J}}^d_C$ over $\Jbar^\vee_C, \Jbar^d_C$; see \cite[Proposition 4.3]{MSY}. We then use singular Riemann--Roch for quotient stacks \cite{EG, EG2} to define the Chow-theoretic Fourier transform
\[
\FF^d = \sum_i\FF^d_i \in \Chow_*(\Jbar_C^\vee \times_B \Jbar^d_C), \quad \FF^d_i \in \Chow_{2g - i + \dim B}(\Jbar_C^\vee \times_B \Jbar^d_C)
\]
and its inverse
\[
(\FF^d)^{-1} = \sum_i(\FF^d)^{-1}_i \in \Chow_*(\Jbar_C^d \times_B \Jbar^\vee_C), \quad (\FF^d)^{-1}_i \in \Chow_{2g - i + \dim B}(\Jbar_C^d \times_B \Jbar^\vee_C);
\]
see \cite[Sections 2.4 and 4.4]{MSY}. Note that here and throughout, all Chow groups are taken with $\BQ$-coefficients. Further applying the cycle class maps
\begin{gather*}
\begin{multlined}[.9\textwidth]
\Chow_{2g - i + \dim B}(\Jbar_C^\vee \times_B \Jbar^d_C) \to H^{\mathrm{BM}}_{4g - 2i + 2\dim B}(\Jbar_C^\vee \times_B \Jbar^d_C, \BQ) \\
\simeq \Hom_{D^b_c(B)}(\pi^\vee_*\BQ_{\Jbar_C^\vee}, \pi_{d*}\BQ_{\Jbar_C^d}[2i - 2g]),
\end{multlined} \\
\begin{multlined}[.9\textwidth]
\Chow_{2g - i + \dim B}(\Jbar_C^d \times_B \Jbar_C^\vee) \to H^{\mathrm{BM}}_{4g - 2i + 2\dim B}(\Jbar_C^d \times_B \Jbar^\vee_C, \BQ) \\
\simeq \Hom_{D^b_c(B)}(\pi_{d*}\BQ_{\Jbar_C^d}, \pi^\vee_{*}\BQ_{\Jbar_C^\vee}[2i - 2g]),
\end{multlined}
\end{gather*}
we obtain the sheaf-theoretic Fourier transform and its inverse
\begin{gather*}
\FF^d = \sum_i\FF^d_i: \pi^\vee_*\BQ_{\Jbar_C^\vee} \to \pi_{d*}\BQ_{\Jbar_C^d}[-], \quad \FF^d_i: \pi^\vee_*\BQ_{\Jbar_C^\vee} \to \pi_{d*}\BQ_{\Jbar_C^d}[2i - 2g] \\
(\FF^d)^{-1} = \sum_i(\FF^d)^{-1}_i: \pi_{d*}\BQ_{\Jbar_C^d} \to \pi^\vee_*\BQ_{\Jbar_C^\vee}[-], \quad (\FF^d)^{-1}_i: \pi_{d*}\BQ_{\Jbar_C^d} \to \pi^\vee_*\BQ_{\Jbar_C^\vee}[2i - 2g];
\end{gather*}
see \cite[Section 2.2.2]{MSY}.

The starting point is the following weaker version of the Fourier vanishing, as compared with the Fourier vanishing for abelian schemes in \cite{DM}; see also \cite[Section 3.4]{MSY}.

\begin{prop}[Fourier vanishing {\cite[Sections 3.5 and 4.4]{MSY}}] \label{prop:halfFV}
For all $i + j < 2g$, we have
\begin{gather*}
\tag{FV1} (\FF^d)^{-1}_j\circ \FF^d_i = 0 \in \Chow_{3g - i - j + \dim B}(\Jbar_C^\vee \times_B \Jbar_C^\vee), \\
\tag{FV2} \FF^d_j \circ (\FF^d)^{-1}_i = 0 \in \Chow_{3g - i - j + \dim B}(\Jbar_C^d \times_B \Jbar_C^d).
\end{gather*}
\end{prop}

Note that only (FV1) was stated and used in \cite{MSY}. Here for our purpose we shall need (FV2) whose proof is identical to that of (FV1); see also \cite[Proposition 2.9]{MSY2}.

We define for $0 \leq k \leq 2g$ Chow self-correspondences
\begin{gather*}
\Fp^d_{\leq k} := \sum_{i \leq k}\FF^d_i \circ (\FF^d)^{-1}_{2g - i} \in \Chow_{g + \dim B}(\Jbar_C^d \times_B \Jbar^d_C), \\
\Fp^\vee_{\leq k} := \sum_{i \leq k} (\FF^d)^{-1}_i \circ \FF^d_{2g - i} \in \Chow_{g + \dim B}(\Jbar^\vee_C \times_B \Jbar^\vee_C).
\end{gather*}
With the weaker Fourier vanishing (FV1) (resp.~(FV2)) we can only conclude that the $\Fp^d_{\leq k}$ (resp.~$\Fp^\vee_{\leq k}$) are \emph{semi-orthogonal} idempotents, \emph{i.e.},
\begin{equation*} \label{eq:semiorth}
\Fp^d_{\leq l} \circ \Fp^d_{\leq k} = \Fp^d_{\leq k}, \quad \Fp^\vee_{\leq l} \circ \Fp^\vee_{\leq k} = \Fp^\vee_{\leq k}, \quad k \leq l. 
\end{equation*}
The same statements hold for the induced sheaf-theoretic idempotents
\begin{gather*}
\Fp^d_{\leq k} := \sum_{i \leq k}\FF^d_i \circ (\FF^d)^{-1}_{2g - i}: \pi_{d*}\BQ_{\Jbar_C^d} \to \pi_{d*}\BQ_{\Jbar_C^d}, \\
\Fp^\vee_{\leq k} := \sum_{i \leq k} (\FF^d)^{-1}_i \circ \FF^d_{2g - i} : \pi^\vee_*\BQ_{\Jbar_C^\vee} \to \pi^\vee_*\BQ_{\Jbar_C^\vee}.
\end{gather*}

We can talk about the image of the sheaf-theoretic $\Fp^d_{\leq k}$ (resp.~$\Fp^\vee_{\leq k}$) since the category~$D^b_c(B)$ is pseudo-abelian by \cite[Lemma 2.24]{CH}. The following result crucially relies on the \emph{full support} property of the pushforward complexes $\pi_{d*}\BQ_{\Jbar_C^d}$ and $\pi^\vee_*\BQ_{\Jbar_C^\vee}$; we will come back to this point later in Section \ref{sec:3}.

\begin{thm}[Realization {\cite[Corollary 4.6(ii)]{MSY}}] \label{thm:realize}
For $0 \leq k \leq 2g$, the natural inclusion $\mathrm{Im}(\Fp^d_{\leq k}) \to \pi_{d*}\BQ_{\Jbar_C^d}$ realizes the perverse truncation
\[
{^\Fp}\tau_{\leq k + \dim B}(\pi_{d*}\BQ_{\Jbar_C^d}) \to \pi_{d*}\BQ_{\Jbar_C^d}
\]
and provides a non-canonical splitting
\begin{equation} \label{eq:split1}
\pi_{d*}\BQ_{\Jbar_C^d} \simeq {^\Fp}\tau_{\leq k + \dim B}(\pi_{d*}\BQ_{\Jbar_C^d}) \oplus {^\Fp}\tau_{\geq k + 1+\dim B}(\pi_{d*}\BQ_{\Jbar_C^d}).
\end{equation}
Similarly, the inclusion $\mathrm{Im}(\Fp^\vee_{\leq k}) \to \pi^\vee_*\BQ_{\Jbar_C^\vee}$ realizes the perverse truncation
\[
{^\Fp}\tau_{\leq k + \dim B}(\pi^\vee_*\BQ_{\Jbar_C^\vee}) \to \pi^\vee_*\BQ_{\Jbar_C^\vee}
\]
and provides a non-canonical splitting
\begin{equation} \label{eq:split2}
\pi^\vee_*\BQ_{\Jbar_C^\vee} \simeq {^\Fp}\tau_{\leq k + \dim B}(\pi^\vee_*\BQ_{\Jbar_C^\vee}) \oplus {^\Fp}\tau_{\geq k + 1+\dim B}(\pi^\vee_*\BQ_{\Jbar_C^\vee}).
\end{equation}
\end{thm}
Here the splitting \eqref{eq:split1} (resp.~\eqref{eq:split2}) is given by the image of
\begin{equation} \label{eq:defq}
\Fq^d_{\geq k + 1} := \mathrm{id}_{\pi_{d*}\BQ_{\Jbar_C^d}} - \Fp^d_{\leq k} = \sum_{i \geq k + 1}\FF^d_i \circ (\FF^d)^{-1}_{2g - i}: \pi_{d*}\BQ_{\Jbar_C^d} \to \pi_{d*}\BQ_{\Jbar_C^d},
\end{equation}
and respectively,
\begin{equation} \label{eq:defq2}
\Fq^\vee_{\geq k + 1} := \mathrm{id}_{\pi^\vee_*\BQ_{\Jbar_C^\vee}} - \Fp^\vee_{\leq k} = \sum_{i \geq k + 1} (\FF^d)^{-1}_i \circ \FF^d_{2g - i} : \pi^\vee_*\BQ_{\Jbar_C^\vee} \to \pi^\vee_*\BQ_{\Jbar_C^\vee}.
\end{equation}


An important outcome of Theorem \ref{thm:realize} is the multiplicativity of the perverse filtration $P_\bullet H^*(\Jbar_C^d, \BQ)$. There is also a stronger, sheaf-theoretic version concerning the cup-product
\begin{equation} \label{eq:cupprod}
\cup: \pi_{d*}\BQ_{\Jbar_C^d} \otimes \pi_{d*}\BQ_{\Jbar_C^d} \to \pi_{d*}\BQ_{\Jbar_C^d}.
\end{equation}

\begin{thm}[Multiplicativity {\cite[Corollary 4.6(iii)]{MSY}}] \label{thm:mult}
For integers $k, l$, the perverse truncation
\[
\cup: {^\Fp}\tau_{\leq k + \dim B}(\pi_{d*}\BQ_{\Jbar_C^d})  \otimes {^\Fp}\tau_{\leq l + \dim B}(\pi_{d*}\BQ_{\Jbar_C^d}) \to \pi_{d*}\BQ_{\Jbar_C^d}.
\]
of the cup-product in \eqref{eq:cupprod} factors through
\[
\cup: {^\Fp}\tau_{\leq k + \dim B}(\pi_{d*}\BQ_{\Jbar_C^d})  \otimes {^\Fp}\tau_{\leq l + \dim B}(\pi_{d*}\BQ_{\Jbar_C^d}) \to {^\Fp}\tau_{\leq k + l + \dim B}(\pi_{d*}\BQ_{\Jbar_C^d}).
\]
In other words, the composition
\begin{multline*}
{^\Fp}\tau_{\leq k + \dim B}(\pi_{d*}\BQ_{\Jbar_C^d})  \otimes {^\Fp}\tau_{\leq l + \dim B}(\pi_{d*}\BQ_{\Jbar_C^d}) \to \pi_{d*}\BQ_{\Jbar_C^d} \otimes \pi_{d*}\BQ_{\Jbar_C^d} \xrightarrow{\cup} \pi_{d*}\BQ_{\Jbar_C^d} \\
\to \tau_{\geq k + l + 1 + \dim B}(\pi_{d*}\BQ_{\Jbar_C^d})
\end{multline*}
is zero.
\end{thm}

As a result, the cup-product naturally descends via \eqref{eq:disttri} to a collection of morphisms
\begin{equation} \label{eq:cupbarkl}
\overline{\cup}: \CH^d_{(k)} \otimes \CH^d_{(l)} \to \CH^d_{(k + l)}.
\end{equation}
Setting the sheaf-theoretic associated graded
\begin{equation} \label{eq:defCH}
\CH^d := \bigoplus_{k = 0}^{2g} \CH^d_{(k)} \in D^b_c(B),
\end{equation}
we obtain a cup-product
\begin{equation} \label{eq:cupbar}
\overline{\cup}: \CH^d \otimes \CH^d \to \CH^d
\end{equation}
making $\CH^d$ a graded ring object in $D^b_c(B)$. Note however that \eqref{eq:cupprod} and \eqref{eq:cupbar} are in general \emph{not} compatible with the non-canonical isomorphism \eqref{eq:BBD}.

The proof of Theorem \ref{thm:mult} involves the convolution product defined via the Chow~class
\[
\FC^d := (\FF^d)^{-1} \circ [\Delta^{\mathrm{sm}}_{\Jbar^d_C/B}] \circ(\FF^d \times \FF^d) \in \Chow_*(\Jbar^\vee_C \times_B \Jbar^\vee_C \times_B \Jbar^\vee_C),
\]
where $\Delta^{\mathrm{sm}}_{\Jbar^d_C/B} \subset \Jbar^d_C \times_B \Jbar^d_C \times_B \Jbar^d_C$ is the relative small diagonal responsible for the cup-product~\eqref{eq:cupprod}. A key step in establishing Theorem \ref{thm:mult} is a dimension bound
\begin{equation} \label{eq:dimbound}
\FC^d \in \Chow_{\leq 2g + \dim B}(\Jbar^\vee_C \times_B \Jbar^\vee_C \times_B \Jbar^\vee_C).
\end{equation}
The proof of the dimension bound further traces back to the support of the convolution kernel
\begin{equation} \label{eq:convker2}
\CK^d:= (\CP^d)^{-1} \circ \CO_{\Delta^{\mathrm{sm}}_{\overline{\CJ}_C^d/B}} \circ(\CP^d \boxtimes \CP^d) \in D^b\Coh(\overline{\CJ}_C^\vee \times_B \overline{\CJ}^\vee_C \times_B \overline{\CJ}^\vee_C)_{(d, d, -d)}.
\end{equation}
Here $\CO_{\Delta^{\mathrm{sm}}_{\overline{\CJ}_C^d/B}}$ is the structure sheaf of the $\mu_r^{\times 3}$-gerbe over $\Delta^{\mathrm{sm}}_{\Jbar^d_C/B}$; see \cite[Section 4.4]{MSY}.

Via the cycle class map
\begin{multline*}
\Chow_{\leq 2g + \dim B}(\Jbar^\vee_C \times_B \Jbar^\vee_C \times_B \Jbar^\vee_C) \to H^{\mathrm{BM}}_{\leq 4g + 2\dim B}(\Jbar^\vee_C \times_B \Jbar^\vee_C \times_B \Jbar^\vee_C, \BQ) \\ \simeq \mathrm{Hom}_{D^b_c(B)}(\pi^\vee_{*}\BQ_{\Jbar^\vee_C} \otimes \pi^\vee_{*}\BQ_{\Jbar^\vee_C}, \pi^\vee_{*}\BQ_{\Jbar^\vee_C}[\geq \!{-2g}])
\end{multline*}
the class $\FC^d$ induces the sheaf-theoretic convolution product
\begin{equation*} \label{eq:convd}
*^d: \pi^\vee_{*}\BQ_{J^\vee_C} \otimes \pi^\vee_{*}\BQ_{J^\vee_C} \to \pi^\vee_{*}\BQ_{J^\vee_C}[\geq \!{-2g}];
\end{equation*}
see \cite[Section 2.2.3]{MSY}. By definition, there is a commutative diagram interchanging the two product structures
\begin{equation} \label{eq:commconv}
\begin{tikzcd}
\pi^\vee_{*}\BQ_{\Jbar^\vee_C} \otimes \pi^\vee_{*}\BQ_{\Jbar^\vee_C} \arrow[r, "*^d"] \arrow[d, "\FF^d \otimes \FF^d"] & \pi^\vee_{*}\BQ_{\Jbar^\vee_C}[\geq \!-2g] \arrow[d, "\FF^d"] \\
\pi_{d*}\BQ_{\Jbar^d_C}[-] \otimes \pi_{d*}\BQ_{\Jbar^d_C}[-] \arrow[r, "\cup"] & \pi_{d*}\BQ_{\Jbar^d_C}[-].
\end{tikzcd}
\end{equation}

We are particularly interested in the lowest codimension ($= \!g$) component of $\FC^d$ in \eqref{eq:dimbound}; we call it the \emph{reduced} convolution class
\begin{equation} \label{eq:convkerbar}
\FC^d_\mathrm{red} \in \Chow_{2g + \dim B}(\Jbar^\vee_C \times_B \Jbar^\vee_C \times_B \Jbar^\vee_C).
\end{equation}
The following proposition is crucial to the proof of the degree-independence Theorem \ref{thm:mainint}(iii).

\begin{prop} \label{prop:trunC}
The reduced convolution class $\FC^d_{\mathrm{red}}$ in \eqref{eq:convkerbar} is independent of the degree~$d$.
\end{prop}

\begin{proof}
Since the convolution kernel $\CK^d$ in \eqref{eq:convker2} is shown in \cite[Proposition 3.2 and Section~4.4]{MSY} to be supported in codimension $g$, the component $\FC^d_{\mathrm{red}}$ is simply given by the fundamental class of the codimension $g$ support of $\CK^d$. On the other hand, \cite[Corollary 4.5]{MSY} shows that \'etale locally the~$\CK^d$ only differ by tensoring a line bundle, and hence have the same support. In particular, the class $\FC^d_{\mathrm{red}}$ is independent of $d$.
\end{proof}

Therefore, $\FC^d_{\mathrm{red}}$ can be abbreviated to $\FC_{\mathrm{red}}$, and induces a reduced convolution product
\begin{equation} \label{eq:trunconv}
*_{\mathrm{red}}: \pi^\vee_{*}\BQ_{J^\vee_C} \otimes \pi^\vee_{*}\BQ_{J^\vee_C} \to \pi^\vee_{*}\BQ_{J^\vee_C}[-2g]
\end{equation}
which is independent of $d$.

We finish with a brief summary of the parallel statements in global cohomology.
\begin{enumerate}
\item[(i)] The cohomological Fourier transforms induce two sequences of semi-orthogonal idempotents $\Fp^d_{\leq k}, \Fp^\vee_{\leq k}$, which give a new description of the perverse filtrations
\begin{gather*}
P_kH^*(\Jbar_C^d, \BQ) = \mathrm{Im}\left(\Fp^d_{\leq k}: H^*(\Jbar_C^d, \BQ) \to H^*(\Jbar_C^d, \BQ)\right), \\
P_kH^*(\Jbar_C^\vee, \BQ) = \mathrm{Im}\left(\Fp^\vee_{\leq k} : H^*(\Jbar_C^\vee, \BQ) \to H^*(\Jbar_C^\vee, \BQ)\right).
\end{gather*}

\item[(ii)] The perverse filtration $P_\bullet H^*(\Jbar_C^d, \BQ)$ is multiplicative with respect to the cup-product, \emph{i.e.}, for integers $k, l$ we have
\[
\cup: P_k H^*(\overline{J}^d_C, \BQ) \otimes P_l H^*(\overline{J}^d_C, \BQ) \to  P_{k + l} H^*(\overline{J}^d_C, \BQ).
\]
In particular, the associated graded
\begin{equation} \label{eq:defBH}
\BH^d:= \bigoplus_{k,m} \mathrm{Gr}^P_kH^m(\overline{J}^d_C, \BQ)
\end{equation}
inherits a cup-product
\begin{equation*} \label{eq:cupbarglobal}
\overline{\cup}: \BH^d \otimes \BH^d \to \BH^d
\end{equation*}
making it a bigraded $H^*(B, \BQ)$-algebra. By definition, elements in $H^m(B, \BQ)$ have bigrading $(0, m)$.

\item[(iii)] There is a commutative diagram
\[
\begin{tikzcd}
H^m(\Jbar^\vee_C, \BQ) \otimes H^n(\Jbar^\vee_C, \BQ) \arrow[r, "*^d"] \arrow[d, "\FF^d \otimes \FF^d"] & H^{\geq m + n - 2g}(\Jbar^\vee_C, \BQ) \arrow[d, "\FF^d"] \\
H^*(\Jbar^d_C, \BQ) \otimes H^*(\Jbar^d_C, \BQ) \arrow[r, "\cup"] & H^*(\Jbar^d_C, \BQ).
\end{tikzcd}
\]
Moreover, the reduced convolution product
\[
*_{\mathrm{red}}: H^m(\Jbar^\vee_C, \BQ) \otimes H^n(\Jbar^\vee_C, \BQ) \to H^{m + n - 2g}(\Jbar^\vee_C, \BQ)
\]
is independent of $d$.
\end{enumerate}

\subsection{Degree-independence}

Our main theorem for compactified Jacobian families of integral locally planar curves is the following.

\begin{thm} \label{thm:mainint}
Let $\pi_d: \Jbar^d_C \to B$ and $\pi^\vee: \Jbar_C^\vee \to B$ be as in Sections \ref{sec:2.1} and \ref{sec:2.2}.
\begin{enumerate}
\item[(i)] (Fourier-stability) For integers $k, l$, the perverse truncation
\[
\FF^d_l: {^\Fp}\tau_{\leq k + \dim B}(\pi_*^\vee\BQ_{\Jbar_C^\vee}) \to \pi_{d*}\BQ_{\Jbar^d_C}[2l - 2g]
\]
is zero if $l < 2g - k$, and factors through
\[
\FF^d_l: {^\Fp}\tau_{\leq k + \dim B}(\pi_*^\vee\BQ_{\Jbar_C^\vee}) \to {^\Fp}\tau_{\leq l + \dim B}(\pi_{d*}\BQ_{\Jbar^d_C})[2l - 2g]
\]
if $l \geq 2g - k$. Similarly, the perverse truncation
\[
(\FF^d)^{-1}_l: {^\Fp}\tau_{\leq k + \dim B}(\pi_{d*}\BQ_{\Jbar_C^d}) \to \pi_*^\vee\BQ_{\Jbar^\vee_C}[2l - 2g]
\]
is zero if $l < 2g - k$, and factors through
\[
(\FF^d)^{-1}_l: {^\Fp}\tau_{\leq k + \dim B}(\pi_{d*}\BQ_{\Jbar_C^d}) \to {^\Fp}\tau_{\leq l + \dim B}(\pi_*^\vee\BQ_{\Jbar^\vee_C})[2l - 2g]
\]
if $l \geq 2g - k$. For $0 \leq k \leq 2g$, $\FF^d_k$ and $(\FF^d)^{-1}_{2g - k}$ induce mutually inverse isomorphisms
\begin{equation} \label{eq:isomH}
\CH^\vee_{(2g - k)}[2g - 2k] \xrightleftharpoons[\overline{(\FF^d)_{2g - k}^{-1}}]{\overline{\FF^d_k}} \CH^d_{(k)}.
\end{equation}

\item[(ii)] (Multiplicativity for $*_{\mathrm{red}}$) For integers $k, l$, the perverse truncation
\[
*_{\mathrm{red}}: {^\Fp}\tau_{\leq k + \dim B}(\pi_*^\vee\BQ_{\Jbar_C^\vee})  \otimes {^\Fp}\tau_{\leq l + \dim B}(\pi_*^\vee\BQ_{\Jbar_C^\vee}) \to \pi_*^\vee\BQ_{\Jbar_C^\vee}[-2g],
\]
of the reduced convolution product in \eqref{eq:trunconv} factors through
\[
*_{\mathrm{red}}: {^\Fp}\tau_{\leq k + \dim B}(\pi_*^\vee\BQ_{\Jbar_C^\vee})  \otimes {^\Fp}\tau_{\leq l + \dim B}(\pi_*^\vee\BQ_{\Jbar_C^\vee}) \to {^\Fp}\tau_{\leq k + l -2g + \dim B}(\pi_*^\vee\BQ_{\Jbar_C^\vee})[-2g],
\]
In particular, the reduced convolution product descends to a collection of morphisms
\begin{equation} \label{eq:trunconvbarkl}
\overline{*}_{\mathrm{red}}: \CH^\vee_{(k)} \otimes \CH^\vee_{(l)} \to \CH^\vee_{(k + l - 2g)}[-2g].
\end{equation}

\item[(iii)] (Degree-independence) The graded ring object $\CH^d \in D^b_c(B)$ in \eqref{eq:defCH} is completely determined by the morphisms $\overline{*}_{\mathrm{red}}$. In particular, for integers $d, d'$, there is an isomorphism
\[
\CH^d \simeq \CH^{d'} \in D^b_c(B)
\]
of graded ring objects.
\end{enumerate}
\end{thm}

\begin{proof}
Essentially all three parts are consequences of the Fourier vanishing. We begin with part (i). In view of the realization Theorem \ref{thm:realize}, for the first statement of (i) it suffices to show
\begin{gather}
\FF^d_l \circ \Fp^\vee_{\leq k} = 0, \quad l < 2g - k, \label{eq:fstable} \\
\Fq^d_{\geq l + 1} \circ \FF^d_l = 0, \label{eq:fstable2}
\end{gather}
where $\Fq^d_{\geq l + 1}$ is defined in \eqref{eq:defq}. Expanding the left-hand side of both equations
\begin{gather*}
\FF^d_l \circ \Fp^\vee_{\leq k} = \sum_{i \leq k} \FF^d_l \circ (\FF^d)^{-1}_i \circ \FF^d_{2g - i}, \\
\Fq^d_{\geq l + 1} \circ \FF^d_l = \sum_{i \geq l + 1} \FF^d_i \circ (\FF^d)^{-1}_{2g - i} \circ \FF^d_l,
\end{gather*}
we see that \eqref{eq:fstable} (resp.~\eqref{eq:fstable2}) follows directly from (FV2) (resp.~(FV1)). The second statement of (i) is parallel.

To prove \eqref{eq:isomH} we first observe that
\[
(\FF^d)^{-1}_{2g - k} \circ \Fp_{\leq k - 1}^d = 0;
\]
hence $(\FF^d)^{-1}_{2g - k}$ descends to a well-defined morphism
\[
\overline{(\FF^d)^{-1}_{2g - k}}: \CH^d_{(k)} \to \CH^\vee_{(2g - k)}[2g - 2k].
\]
Similarly, $\FF^d_k$ descends to a well-defined morphism
\[
\overline{\FF^d_k}: \CH^\vee_{(2g - k)}[2g - 2k] \to \CH^d_{(k)}.
\]
To show that they are inverse to each other, we consider the identities
\begin{align}
\Fp_{\leq k}^d & = \FF^d \circ (\FF^d)^{-1} \circ \Fp_{\leq k}^d = \sum_{i \geq 2g - k} \FF^d_{2g - i} \circ (\FF^d)^{-1}_i \circ \Fp_{\leq k}^d \nonumber \\
& = \FF^d_k \circ (\FF^d)^{-1}_{2g - k} \circ \Fp_{\leq k}^d + \sum_{i > 2g - k} \FF^d_{2g - i} \circ (\FF^d)^{-1}_i \circ \Fp_{\leq k}^d \label{eq:expand}
\end{align}
where the second identity holds both by (FV1) and for dimension reasons. Now the condition~$i > 2g - k$
translates to $2g - i < k$, which implies that the entire second term in \eqref{eq:expand} factors through
\[
{^\Fp}\tau_{\leq k + \dim B}(\pi_{d*}\BQ_{\Jbar_C^d}) \to {^\Fp}\tau_{\leq k - 1 + \dim B}(\pi_{d*}\BQ_{\Jbar_C^d}).
\]
Therefore we have
\[
\mathrm{id} = \overline{\FF^d_k} \circ \overline{(\FF^d)^{-1}_{2g - k}}: \CH^d_{(k)} \to \CH^d_{(k)}.
\]
The other direction is parallel. This proves (i).

For (ii) we consider via \eqref{eq:commconv} the identity
\[
*^d = (\FF^d)^{-1} \circ \cup \circ (\FF^d \otimes \FF^d).
\]
For dimension reasons we then have
\begin{equation} \label{eq:convexpand}
*_{\mathrm{red}} = \sum_i\sum_j(\FF^d)^{-1}_{2g - i - j} \circ \cup \circ (\FF^d_i \otimes \FF^d_j).
\end{equation}
Again, in view of the realization Theorem \ref{thm:realize}, to prove (ii) it suffices to show
\[
\Fq^\vee_{\geq k + l + 1 - 2g} \circ *_{\mathrm{red}} \circ (\Fp^\vee_{\leq k} \otimes \Fp^\vee_{\leq l}) = 0
\]
where $\Fq^\vee_{\geq k + l + 1 - 2g}$ is defined in \eqref{eq:defq2}. We expand $*_{\mathrm{red}}$ using \eqref{eq:convexpand}, and find
\begin{align*}
& \Fq^\vee_{\geq k + l + 1 - 2g} \circ *_{\mathrm{red}} \circ (\Fp^\vee_{\leq k} \otimes \Fp^\vee_{\leq l}) \\
{}= & \sum_i\sum_j \Fq^\vee_{\geq k + l + 1 - 2g} \circ (\FF^d)^{-1}_{2g - i - j} \circ \cup \circ \left((\FF^d_i \circ \Fp^\vee_{\leq k}) \otimes (\FF^d_j \circ \Fp^\vee_{\leq l})\right) \\
{}= & \sum_{i \geq 2g - k}\sum_{j \geq 2g - l} \Fq^\vee_{\geq k + l + 1 - 2g} \circ (\FF^d)^{-1}_{2g - i - j} \circ \cup \circ \left((\FF^d_i \circ \Fp^\vee_{\leq k}) \otimes (\FF^d_j \circ \Fp^\vee_{\leq l})\right)
\end{align*}
where the second identity uses (FV2). This time the conditions $i \geq 2g - k$ and $j \geq 2g - l$ translate~to
\[
2g - i - j \leq k + l - 2g.
\]
But then we have
\[
\Fq^\vee_{\geq k + l + 1 - 2g} \circ (\FF^d)^{-1}_{2g - i - j} = 0
\]
for such $i, j, k, l$ by (FV2). This proves (ii).

For (iii) we consider the following diagram
\begin{equation} \label{eq:diagram}
\begin{tikzcd}
\CH^\vee_{(2g - k)}[2g - 2k] \otimes \CH^\vee_{(2g - l)}[2g - 2l] \arrow[r, "\overline{*}_{\mathrm{red}}"] \arrow[d, shift left=1pt, harpoon, "\overline{\FF^d_k} \otimes \overline{\FF^d_l}"] & \CH^\vee_{(2g - k - l)}[2g - 2k - 2l] \arrow[d, shift left=1pt, harpoon, "\overline{\FF^d_{k + l}}"] \\
\CH^d_{(k)} \otimes \CH^d_{(l)} \arrow[r, "\overline{\cup}"] \arrow[u, shift left=1pt, harpoon, "\overline{(\FF^d)_{2g - k}^{-1}} \otimes \overline{(\FF^d)_{2g - l}^{-1}}"] & \CH^d_{(k + l)}, \arrow[u, shift left=1pt, harpoon, "\overline{(\FF^d)_{2g - k - l}^{-1}}"]
\end{tikzcd}
\end{equation}
where the top and bottom rows are given by \eqref{eq:trunconvbarkl} and \eqref{eq:cupbarkl}, respectively, and the vertical arrows are isomorphisms by \eqref{eq:isomH}. We claim that the diagram commutes. Indeed, 
decomposing
\[
*^d = \sum_{c \geq g} *^d_c, \quad *^d_c: \pi^\vee_{*}\BQ_{J^\vee_C} \otimes \pi^\vee_{*}\BQ_{J^\vee_C} \to \pi^\vee_{*}\BQ_{J^\vee_C}[{2c - 4g}]
\]
according to the codimension, we have $*^d_g = *_{\mathrm{red}}$. There are identities
\begin{align}
\cup \circ (\Fp^d_{\leq k} \otimes \Fp^d_{\leq l}) & = \FF^d \circ *^d \circ \left(((\FF^d)^{-1} \circ \Fp^d_{\leq k}) \otimes ((\FF^d)^{-1} \circ \Fp^d_{\leq l})\right) \nonumber\\
& = \sum_{c \geq g}\sum_{i \geq 2g - k}\sum_{j \geq 2g - l} \FF^d_{5g - i - j - c} \circ *^d_c \circ \left(((\FF^d)^{-1}_i \circ \Fp^d_{\leq k}) \otimes ((\FF^d)^{-1}_j \circ \Fp^d_{\leq l})\right) \label{eq:expandcommute}
\end{align}
where the second identity uses both (FV1) and dimension constraints. We observe that \eqref{eq:expandcommute} contains the term
\[
\FF^d_{k + l} \circ *_{\mathrm{red}} \circ \left(((\FF^d)_{2g - k}^{-1} \circ \Fp^d_{\leq k}) \otimes ((\FF^d)_{2g - l}^{-1} \circ \Fp^d_{\leq l})\right)
\]
and all the other terms satisfy
\[
5g - i - j - c < k + l.
\]
By (i), this means that all the other terms in \eqref{eq:expandcommute} factor through
\[
{^\Fp}\tau_{\leq k + \dim B}(\pi_{d*}\BQ_{\Jbar_C^d}) \otimes {^\Fp}\tau_{\leq l + \dim B}(\pi_{d*}\BQ_{\Jbar_C^d}) \to {^\Fp}\tau_{\leq k + l - 1 + \dim B}(\pi_{d*}\BQ_{\Jbar_C^d})
\]
and hence descend to zero. This shows the commutativity of the diagram \eqref{eq:diagram}.

Finally, for any integer $d$, we conclude from (i) that there are mutually inverse graded isomorphisms
\begin{equation}\label{H_tilde}
\widetilde{\CH}^\vee:=\bigoplus_{k = 0}^{2g}\CH^\vee_{(2g - k)}[2g - 2k] \xrightleftharpoons[\oplus \overline{(\FF^d)_{2g - k}^{-1}}]{\oplus \overline{\FF^d_k}} \bigoplus_{k = 0}^{2g}\CH^d_{(k)} = \CH^d.
\end{equation}
Moreover, by the commutative diagram \eqref{eq:diagram}, the above isomorphisms are compatible with the reduced convolution product $\overline{*}_{\mathrm{red}}$ on the left-hand side, and the cup-product $\overline{\cup}$ on the right-hand side. This proves (iii) since by Proposition \ref{prop:trunC}, the product $\overline{*}_{\mathrm{red}}$ on the left-hand side is independent of $d$. The proof of Theorem \ref{thm:mainint} is now complete.
\end{proof}

We deduce immediately the $d$-independence of the associated graded $\BH^d$ in \eqref{eq:defBH} by taking global cohomology.

\begin{cor} \label{cor:mainint}
For integers $d, d'$, there is an isomorphism
\[
\BH^d \simeq \BH^{d'}
\]
of bigraded $H^*(B, \BQ)$-algebras.
\end{cor}

\begin{rmk} \label{rem:chow}
It is natural to ask if Corollary \ref{cor:mainint} also holds at the Chow level. The question is whether there is a \emph{canonically defined} Chow-theoretic perverse filtration $P_\bullet\Chow^*(\Jbar_C^\vee)$, which in particular satisfies
\[
P_k\Chow^*(\Jbar_C^\vee) = \mathrm{Im}\left(\Fp^\vee_{\leq k}: \Chow^*(\Jbar_C^\vee) \to \Chow^*(\Jbar_C^\vee)\right)
\]
regardless of the degree $d$. The answer is expected to be yes assuming certain Bloch--Beilinson type motivic conjectures in the relative setting.
\end{rmk}

\section{Fine compactified Jacobians} \label{sec:3}

As we see in Section \ref{sec:2}, the main geometric ingredients in treating the associated graded with respect to the perverse filtration are 
\begin{enumerate}
    \item[(i)] Ng\^o's full support theorem;
    \item[(ii)] a theory of Fourier transforms.
\end{enumerate}
In the first part of this section we describe a general framework handling the cohomology ring of a dualizable abelian fibration, extending \cite{MSY}. Then in the second part we apply this to fine compactified Jacobians of reduced locally planar curves both in the global and the local settings, and prove Theorems \ref{thm0.7} and \ref{thm0.8} respectively.

\begin{rmk}
    Strictly speaking, the cases discussed in Section \ref{sec:2} fall outside the framework of dualizable abelian fibrations. Instead, they can be viewed as ``twisted'' variants. Because these sections focus on degree independence, where normalizations and gerbes play a crucial role, we address them explicitly and do not incorporate them into the general framework. This clarifies the main ideas while also helping to avoid overly heavy notation in the framework.
\end{rmk}

\subsection{Dualizable abelian fibrations} We recall the notion of dualizable abelian fibration following \cite{MSY}.

We say that $\pi: M \to B$ is an \emph{abelian fibration} if both $M$ and $B$ are nonsingular varieties, $\pi$ is flat and proper, and the restriction of $M$ to a certain nonempty open subset
\[
\pi_U: M_U \to U \subset B 
\]
is an abelian scheme. 

We say that $\pi^\vee: M^\vee \to B$ is dual to the abelian fibration $\pi: M\to B$, if $\pi^\vee: M^\vee \to B$ is an abelian fibration and there exists an open subset $U\subset B$ over which $\pi$ and $\pi^\vee$ form dual abelian schemes. We say that 
\[
{\CP} \in D^b\mathrm{Coh}(M^\vee\times_B M)
\]
is a \emph{normalized Poincar\'e complex}, if the restriction of $\CP$ to $M_U^\vee \times_U M_U$ recovers the normalized Poincar\'e line bundle $\CL$. Here $M_U, M_U^\vee$ are dual abelian schemes over some open $U\subset B$.

\begin{defn}[Dualizable abelian fibration \cite{MSY}]\label{def:DAB}
A dualizable abelian fibration consists of the following data
\[
(M, B, M^\vee, \CP, \CK).
\]
Here $\pi: M \to B$ is an abelian fibration of relative dimension $g$ with a dual abelian fibration $\pi^\vee: M^\vee \to B$ satisfying the following.

\begin{enumerate}
    \item[(a1)] (Poincar\'e) There is a normalized Poincar\'e complex~${\CP} \in D^b\mathrm{Coh}(M^\vee \times_B M)$ which admits an inverse ${\CP}^{-1} \in D^b\mathrm{Coh}(M \times_B M^\vee)$, \emph{i.e.},
    \begin{equation*} \label{eq:dual}
    {\CP}^{-1} \circ {\CP} \simeq \CO_{\Delta_{M^\vee/B}}, \quad   {\CP}\circ \CP^{-1} \simeq \CO_{\Delta_{M/B}}.
    \end{equation*}
    \item[(a2)] (Convolution) There is an object $\CK \in D^b\mathrm{Coh}(M^\vee \times_B M^\vee \times_B M^\vee)$ supported in codimension $\geq g$, which satisfies
    \begin{equation}\label{convolution}
    \CP \circ \CK \simeq \CO_{\Delta^\mathrm{sm}_{M/B}}\circ ({\CP} \boxtimes {\CP} ).
    \end{equation}
    Here $\Delta^\mathrm{sm}_{M/B} \subset M\times_B M \times_B M$ is the small relative diagonal, and $\CO_{\Delta^\mathrm{sm}_{M/B}}\circ ({\CP} \boxtimes {\CP} )$ stands for composing $\CP$ with $\CO_{\Delta^\mathrm{sm}_{M/B}}$ via the first (resp.~second) factor of \mbox{$M \times_B M \times_B M$}. We refer to $\CK$ as the \emph{convolution kernel}.
    \item[(b)] (Support) Both morphisms $\pi, \pi^\vee$ have full support, \emph{i.e.}, every simple perverse sheaf that appears in the pushforward complex $\pi_*\BQ_M$ or $\pi^\vee_*\BQ_{M^\vee}$ has support $B$.
\end{enumerate}
\end{defn}

We denote by $\mathrm{FM}_\CP: D^b\mathrm{Coh}(M^\vee)\to D^b\mathrm{Coh}(M)$ the Fourier--Mukai transform associated with $\CP$. Then the object $\CK$ in (a2) induces a convolution product 
\[
\ast: D^b\mathrm{Coh}(M^\vee) \times D^b\mathrm{Coh}(M^\vee) \to D^b\mathrm{Coh}(M^\vee), \quad F \ast G = p_{3*}( p_1^*F \otimes p_2^* G \otimes \CK)
\]
where the $p_i: M^\vee \times_B M^\vee \times_B M^\vee \to M^\vee$ are the three projections. We have by definition
\begin{equation*}\label{FM_Conv}
\mathrm{FM}_{\CP}(F \ast G) \simeq \mathrm{FM}_{\CP}(F) \otimes\mathrm{FM}_{\CP}(G) ,\quad F, G \in D^b\mathrm{Coh}(M^\vee).
\end{equation*}

\begin{rmk}
We refer to \cite[Section 1]{MSY} for more details and explanations on dualizable abelian fibrations. Note that there is a minor difference between the definition above and the one in~\cite{MSY}. The condition (b) in \cite[Definition 1.2]{MSY} only requires that $\pi$ has full support; here for our purpose we need both maps $\pi, \pi^\vee$ to have full support in order to relate the projection operators on each side with the perverse filtration.
\end{rmk}

Following \cite[Section 2.4]{MSY} (see also Section \ref{sec:2.2}), the Poincar\'e complex $\CP$ and its inverse~$\CP^{-1}$ induce the Chow-theoretic Fourier transform
\[
\FF = \sum_i\FF_i \in \Chow_*(M^\vee \times_B M), \quad \FF_i \in \Chow_{2g - i + \dim B}(M^\vee \times_B M)
\]
and its inverse
\[
\FF^{-1} = \sum_i\FF^{-1}_i \in \Chow_*(M \times_B M^\vee), \quad \FF^{-1}_i \in \Chow_{2g - i + \dim B}(M \times_B M^\vee);
\]
they also act sheaf-theoretically as in Section \ref{sec:2.2}.

\subsection{Cup-product, Fourier vanishing, and convolution}

We define the shifted perverse sheaves $\CH_{(k)}$ associated with the decomposition of the pushforward complex $\pi_*{\BQ_M}$ as in (\ref{eq:BBD}); a key observation of \cite{MSY} is that the compatibility of the perverse truncation functors and the cup-product (a.k.a.~multiplicitivity) for dualizable abelian fibrations is governed by certain Fourier vanishing results.

\begin{defn} [Fourier vanishing] 
The conditions (FV1) and (FV2) refer to the following statements: for all $i + j < 2g$, we have
\begin{gather*}
\tag{FV1} \FF^{-1}_j\circ \FF_i = 0 \in \Chow_{3g - i - j + \dim B}(M^\vee \times_B M^\vee), \\
\tag{FV2} \FF_j \circ \FF^{-1}_i = 0 \in \Chow_{3g - i - j + \dim B}(M \times_B M).
\end{gather*}
\end{defn}

By \cite[Theorem 2.6(iii)]{MSY}, if the dualizable abelian fibration satisfies (FV1), then the multiplicitivity holds for the perverse truncation functors. In particular, we obtain the following.

\begin{thm}\label{thm3.4}
    Assume that the dualizable abelian fibration
    \[
    (M, B, M^\vee, \CP, \CK)
    \]
    satisfies \textup{(FV1)}. Then the cup-product on $\pi_*\BQ_M$ induces a cup-product on the associated graded
    \[ 
\overline{\cup}: \CH_M \otimes \CH_M \to \CH_M, \quad \CH_M := \bigoplus_{k = 0}^{2g} \CH_{(k)} \in D^b_c(B).\]
In particular, we may view $\CH_M \in D^b_c(B)$ as a graded ring object endowed with $\overline{\cup}$.\end{thm}

Theorem \ref{thm3.4} only needs (FV1) on $M^\vee\times_BM^\vee$. The argument in Section \ref{sec:2} shows that, if we also have the other (FV2), then we may describe the graded ring object $\CH_M$ associated with $M$ using the convolution kernel $\CK$ on $M^\vee \times_B M^\vee \times_B M^\vee$.

More precisely, we consider the Chow-theoretic convolution class
\[
\FC := \FF^{-1} \circ [\Delta^{\mathrm{sm}}_{M/B}] \circ(\FF \times \FF) \in \Chow_*(M^\vee \times_B M^\vee \times_B M^\vee).
\]
Using \eqref{convolution} and singular Riemann--Roch (see \cite[Lemma 2.8]{MSY}), we also have
\begin{equation*}
\FC = \td(-p^*_{12}T_{M^\vee \times_B M^\vee}) \cap \tau(\CK)
\end{equation*}
where $p_{12}: M^\vee \times_B M^\vee \times_B M^\vee \to M^\vee \times_B M^\vee$ is the projection to the first two factors, $T_{M^\vee \times_B M^\vee}$ is the virtual tangent bundle of the l.c.i.~scheme $M^\vee \times_B M^\vee$, and $\tau(-)$ is the Baum--Fulton--MacPherson tau-class. In particular, the dimension bound in Definition~\ref{def:DAB}(a2) implies
\[
\FC \in \Chow_{\leq 2g + \dim B}(M^\vee \times_B M^\vee \times_B M^\vee).
\]
We then consider the codimension $g$ component of $\FC$:
\begin{equation}\label{conv_cycle} 
\FC_{\mathrm{red}} \in \Chow_{2g + \dim B}(M^\vee \times_B M^\vee \times_B M^\vee)
\end{equation}
which we call the \emph{reduced} convolution class.

\begin{thm}\label{thm:DAF}
Let $(M,B,M^\vee,\CP, \CK)$ be a dualizable abelian fibration satisfying both \textup{(FV1)} and \textup{(FV2)}. Then the graded ring object $\CH_M \in D^b_c(B)$ is completely determined by the reduced convolution class $\FC_{\mathrm{red}}$ in \eqref{conv_cycle}. In particular, given two dualizable abelian fibrations
\[
(M, B, M^\vee, \CP, \CK), \quad (M', B, M^\vee, \CP', \CK')
\]
satisfying \textup{(FV1)} and \textup{(FV2)}, we have two convolution classes
\[
\FC , \FC' \in \Chow_{\leq 2g + \dim B}(M^\vee \times_B M^\vee \times_B M^\vee )\]
induced by $\CK, \CK'$ respectively. If we have a match for the reduced parts
\begin{equation*}\label{reduced_eq}
\FC_{\mathrm{red}} = \FC'_{\mathrm{red}} \in \Chow_{2g+\dim B}(M^\vee \times_B M^\vee \times_B M^\vee ),
\end{equation*}
then there is an isomorphism
\[
\CH_M \simeq \CH_{M'} \in D^b_c(B)
\]
of graded ring objects.
\end{thm}

\begin{proof}
    The proof is essentially identical to the one given for compactified Jacobians associated with integral curves in Section \ref{sec:2}. For a dualizable abelian fibration
    \[
    (M, B, M^\vee, \CP, \CK),
    \]
    we show that the graded ring object $\CH_M$ is isomorphic via the Fourier transform to $\widetilde{\CH}_{M^\vee}$ endowed with the convolution product induced by $\FC_{\mathrm{red}}$, as in (\ref{H_tilde}). The proof of Theorem \ref{thm:mainint} works \emph{verbatim}, since the only geometric inputs required are 
    \begin{enumerate}
        \item[(i)] the dimension bound for the support of the convolution kernel $\CK$,
        \item[(ii)] the Fourier vanishing, and
        \item[(iii)] the full support property of $\pi, \pi^\vee$,
    \end{enumerate}
    all of which remain valid in the general setting. The rest of the argument is entirely formal.
\end{proof}

\begin{rmk}
Under the assumption of Theorem \ref{thm:DAF}, statements parallel to Theorem \ref{thm:mainint} hold. Our result also shows that the induced cup-product on the associated graded of the cohomology with respect to the perverse filtration
\[
\BH_M:= \bigoplus_{k,m}\mathrm{Gr}_k^PH^m(M, \BQ)
\]
only relies on the reduced convolution class $\FC_{\mathrm{red}}$ on $M^\vee\times_BM^\vee\times_B M^\vee$. As in Remark \ref{rem:chow}, the parallel statements at the Chow level is conditional on certain Bloch--Beilinson type motivic conjectures.
\end{rmk}

In the next sections, we apply Theorem \ref{thm:DAF} to fine compactified Jacobians over the moduli space $\overline{\CM}_{g,n}$ of stable marked curves, and for versal deformations of a reduced locally planar curve, which completes the proofs of Theorems \ref{thm0.7} and \ref{thm0.8}.

\subsection{Relative fine compactified Jacobians}

Let $C \to B$ be a flat projective family of connected reduced locally planar curves of arithmetic genus $g$ over a nonsingular base variety~$B$, with a section~\mbox{$s: B \to C$} through the smooth locus of $C \to B$. Let $g: J^{\underline{0}}_C \to B$ be the relative Picard space parameterizing pairs
\[
(C_b, L_b),  \quad L_b \in J^{\underline{0}}_{C_b},
\]
where $C_b \subset C$ is the fiber over a closed point $b\in B$ and $L_b$ is a multidegree $0$ line bundle on~$C_b$. The commutative group scheme $J^{\underline{0}}_{C_b}$ is irreducible, admitting a Chevalley decomposition
\[
0 \to R_b \to J^{\underline{0}}_{C_b} \to A_b \to 0
\]
with $R_b$ an commutative affine group and $A_b$ an abelian variety. This defines a constructible function
\[
\delta: B \to \BZ_{\geq 0}, \quad b \mapsto \dim R_b.
\]
For an irreducible closed subscheme $Z\subset B$, we define $\delta(Z)$ to be $\delta(b)$ with $b\in Z$ a general closed point; alternatively, we have
\[
\delta(Z) = \mathrm{min}\{\delta(b),\,  b\in Z\}.
\]
Following \cite{Ngo}, we say that $C \to B$ is \emph{$\delta$-regular}, if we have
\[
\delta_Z \leq \mathrm{codim}_B(Z)  
\]
for any irreducible closed $Z\subset B$. 


To study fine compactified Jacobians in the relative setting, we consider the polarization stability conditions of Esteves \cite{Est}; this specializes to the stability conditions of Kass--Pagani~\cite{KP} and Melo~\cite{Melo} for the universal curve over $\overline{\CM}_{g,n}$, and the stability conditions considered in Migliorini--Shende--Viviani~\cite{MSV} for versal deformations of a reduced locally planar curve.

Since the explicit form of a stability condition is not important for our purpose, in the following we only describe the definitions of stability conditions and relative fine compactified Jacobians briefly; we refer to \cite{MRV0} for more details. We mainly need the existence of relative fine compactified Jacobians and some properties which we will emphasize.

Recall that a polarization $\phi_b$ of degree $d$ on a connected reduced curve $C_b$ is an assignment of a rational number to each irreducible component of $C_b$, whose total sum over all irreducible components is $d$. A polarization as above yields a stability condition for $C_b$, which allows us to consider stable and semistable rank $1$ torsion-free sheaves on $C_b$ via inequalities with respect to $\phi_b$. A stability condition $\phi_b$ is called \emph{nondegenerate} if there are no strictly semistable sheaves. In this case, the \emph{fine compactified Jacobian} $\overline{J}^{\phi_b}_{C_b}$ is the moduli space of $\phi_b$-stable rank $1$ torsion-free sheaves on $C_b$. For two different nondegenerate stability conditions $\phi_b, \phi'_b$ associated with the same curve $C_b$, the fine compactified Jacobians $\overline{J}^{\phi_b}_{C_b}, \overline{J}^{\phi'_b}_{C_b}$ are birational but are not isomorphic in general; see \cite[Theorem B]{MRV0}.

Now we consider the relative case. A polarization stability condition for the family $C\to B$ is a polarization on each closed fiber $C_b$ which is compatible with specializations \cite[Section~5]{MRV0}. Once we fix a polarization stability condition
\[
\phi:= \{\phi_b,\, b\in B\}
\]
for $C\to B$ which is nondegenerate (\emph{i.e.}, each $\phi_b$ is nondegenerate), we can consider the \emph{relative} fine compactified Jacobian
\[
\pi: \Jbar^\phi_C \to B
\]
parameterizing $(C_b, F_b)$ with $F_b$ a $\phi_b$-stable rank $1$ torsion-free sheaf on $C_b$.

\begin{rmk}
Strictly speaking, the relative fine compactified Jacobian $\Jbar^\phi_C$ may be an algebraic space. But for the purpose of proving Theorem \ref{thm0.8} one can always perform an \'etale base change over $B$ so that $\Jbar^\phi_C$ stays a scheme.
\end{rmk}

For the rest of this section, we assume that we are given a family of curves $C\to B$ as above, and two nondegenerate (polarization) stability conditions $\phi,\phi'$ of degree $0$, satisfying
\begin{enumerate}
    \item[(i)] the family $C\to B$ is $\delta$-regular, and 
    \item[(ii)] the total spaces of the associated relative fine compactified Jacobians
    \[
    \pi: \overline{J}^\phi_C \to B,\quad  \pi': \overline{J}^{\phi'}_C \to B
    \]
    are nonsingular varieties.
\end{enumerate}

We note that $\overline{J}_C^\phi$ contains a Zariski dense open subset $J_C^\phi$ parameterizing line bundles. The normalized (with respect to the section $s: B \to C$) Poincar\'e line bundle $\CL$ is well-defined over 
\[
J^{\phi'}_C \times_B \overline{J}^\phi_C \cup \overline{J}^{\phi'}_C \times_B J^\phi_C,
\]
and it can be extended to a maximal Cohen--Macaulay sheaf 
\[
\CP:=  j_* \CL \in \mathrm{Coh}(\overline{J}_C^{\phi'} \times_B \overline{J}_C^\phi)
\]
via the open embedding
\[
j: J^{\phi'}_C \times_B \overline{J}^\phi_C \cup \overline{J}^{\phi'}_C \times_B J^\phi_C \hookrightarrow \overline{J}_C^{\phi'} \times_B \overline{J}_C^\phi; 
\]
this was proven in \cite{MRV2} which was built on Arinkin's construction \cite{A2}.

\begin{thm}[\cite{A2, MRV2}]\label{thm:FM_P}
The Fourier--Mukai transform
\[
\mathrm{FM}_\CP: D^b\mathrm{Coh}(\overline{J}^{\phi'}_C) \rightarrow D^b\mathrm{Coh}(\overline{J}^\phi_C)
\]
is an equivalence, whose inverse
\[
\mathrm{FM}_{\CP^{-1}}: D^b\mathrm{Coh}(\overline{J}^{\phi}_C) \rightarrow D^b\mathrm{Coh}(\overline{J}^{\phi'}_C)
\]
is induced by 
\[
\CP^{-1}:= \CP^\vee \otimes p_{\overline{J}_C^\phi}^* \omega_{\overline{J}^{\phi}_C/B}[g] \in D^b\mathrm{Coh}(\overline{J}_C^\phi \times_B \overline{J}_C^{\phi'} ).
\]
Here $p_{\overline{J}_C^\phi}: \overline{J}_C^\phi \times_B \overline{J}_C^{\phi'} \to \overline{J}_C^\phi$ is the projection.
\end{thm}

For the triple $(\overline{J}^\phi_C, B, \overline{J}^{\phi'}_C)$, Theorem \ref{thm:FM_P} guarantees (a1) of Definition \ref{def:DAB} with the Poincar\'e complex given by the Fourier--Mukai kernel $\CP$. For (a2), we also need to consider the convolution kernel 
\begin{equation}\label{convolution11}
\CK:= \CP^{-1} \circ \CO_{\Delta_{\overline{J}^{\phi}_C/B}^{\mathrm{sm}}} \circ (\CP \boxtimes \CP) \in D^b\mathrm{Coh}(\overline{J}_C^{\phi'} \times_B \overline{J}_C^{\phi'}\times_B \overline{J}_C^{\phi'})
\end{equation}
associated with $\CP$, where $\Delta_{\overline{J}^{\phi}_C/B}^{\mathrm{sm}}$ is the relative small diagonal. 

Next, we describe the support of $\CK$. Let $B^\circ \subset B$ be the largest open subset such that the restriction of $C\to B$, denoted by $C^\circ \to B^\circ$, is smooth. Then the restriction of both $\pi, \pi'$ to~$B^\circ$ are identical to the Jacobian fibration
\[
\pi^\circ: J^0_{C^\circ} \to B^\circ.
\]
The restriction of the convolution kernel $\CK$ over $B^\circ$ is the structure sheaf of the relative graph~$\Gamma$ of the group law of $\pi^\circ$:
\[
\Gamma:=\{(C_b, L_1, L_2, L_3) \in J^0_{C^\circ} \times_B J^0_{C^\circ} \times_B J^0_{C^\circ},\, b\in B^\circ,\, L_1, L_2, L_3 \in \mathrm{Pic}^0(C_b),\, L_1\otimes L_2 = L_3\}.
\]
The following proposition is a slight extension of Arinkin's support property \cite[Proposition~3.2]{MSY}. It plays the role of (the easier) Proposition \ref{prop:trunC} in the case of integral curves.

\begin{prop}[Reduced convolution class]\label{prop:supp}
The support of the convolution kernel
\[
\mathrm{Supp}(\CK) \subset \overline{J}_C^{\phi'} \times_B \overline{J}_C^{\phi'}\times_B \overline{J}_C^{\phi'}.
\]
is of codimension $g$. Moreover, the Zariski closure  
\[
\overline{\Gamma} \subset \overline{J}_C^{\phi'} \times_B \overline{J}_C^{\phi'}\times_B \overline{J}_C^{\phi'}
\]
is the only codimension $g$ irreducible component of $\mathrm{Supp}(\CK)$; in particular, the codimension $g$ support of $\CK$ is independent of the stability condition $\phi$.
 \end{prop}

\begin{proof}
The proof is parallel to that of \cite[Proposition 3.2]{MSY}; the arguments there need to be generalized from integral curves to reduced curves, and this was already worked through in~\cite{MRV2} (in order to establish Theorem \ref{thm:FM_P}). For the reader's convenience, in the following, we explain in more detail that the $\delta$-regularity assumption not only guarantees that the support of $\CK$ is of codimension $g$ as in \cite[Proposition 3.2]{MSY}, but also determines the support completely as stated here in the proposition.

We consider 
\[
\CK_3: = p_{123 *}\left( p_{14}^* \CP \otimes p_{24}^*\CP \otimes p_{34}^*\CP\right) \in D^b\mathrm{Coh}(\overline{J}_C^{\phi'} \times_B \overline{J}_C^{\phi'}\times_B \overline{J}_C^{\phi'})
\]
where the $p_{ijk}, p_{ij}$ are the natural projections from $\overline{J}_C^{\phi'} \times_B \overline{J}_C^{\phi'}\times_B \overline{J}_C^{\phi'} \times_B \overline{J}_C^\phi$. The convolution kernel (\ref{convolution11}) is governed by $\CK_3$; it suffices to show that the restriction of $\CK_3$ over $B\setminus B^\circ$ has codimension $>g$.

We consider $b\in B\setminus B^\circ$ (in particular, $\delta(b) >0$), and a point
\[
(F_1,F_2,F_3) \in Z_b:= \mathrm{Supp}(\CK_3) \cap (\overline{J}_{C_b}^{\phi'_b} )^{\times 3}. 
\]
By an identical argument as in the proof of \cite[Proposition 3.2]{MSY}, we have
\begin{equation}\label{smooth_id}
\bigotimes_{i=1}^3 ({F_i}|_{C^{\mathrm{reg}}_b} ) \simeq \CO_{C^{\mathrm{reg}}_b}
\end{equation}
where $C^{\mathrm{reg}}_b \subset C_b$ is the regular locus; we note that \cite[Proposition 3.2]{MSY} only handles integral locally planar curves, but the proof of \cite[Proposition 6.3]{MRV2} generalizes all the relevant arguments to reduced ones.

Next, we recall the natural action 
\[
\mu_b: J^{\underline{0}}_{C_b} \times (\overline{J}_{C_b}^{\phi'_b} )^{\times 3} \to (\overline{J}_{C_b}^{\phi'_b} )^{\times 3}, \quad (L, F_1, F_2, F_3)\mapsto (L\otimes F_1, L\otimes F_2, L\otimes F_3).
\]
It is a smooth morphism of relative dimension $g$. Hence we have
\begin{equation}\label{codim1}
\mathrm{codim}_{(\overline{J}_{C_b}^{\phi'_b} )^{\times 3}}\left( Z_b\right) = \mathrm{codim}_{J^{\underline{0}}_{C_b} \times (\overline{J}_{C_b}^{\phi'_b} )^{\times 3}}\left( \mu_b^{-1}(Z_b)\right).
\end{equation}
Let
\[
\sigma_b: \mu_b^{-1}(Z_b) \to (\overline{J}_{C_b}^{\phi'_b} )^{\times 3}
\]
be the composition of the natural inclusion 
$\mu_b^{-1}(Z_b) \hookrightarrow J_{C_b}^{\underline{0}}\times (\overline{J}_{C_b}^{\phi'_b} )^{\times 3}$, and the projection to the second factor; every closed fiber of $\sigma_b$ lies in the irreducible group variety $J^{\underline{0}}_{C_b}$. By the proof of~\cite[Corollary 7.6]{A2}, (\ref{smooth_id}) implies that any closed fiber of $\sigma_b$ has dimension $\leq \delta(b)$, while a fiber over any closed point lying in the Zariski dense open subset
\[
(J_{C_b}^{\phi'_b})^{\times 3} \subset (\overline{J}_{C_b}^{\phi'_b} )^{\times 3}
\]
is $0$-dimensional. In particular, for each $b\in B$ we have
\[
\mathrm{codim}_{J^{\underline{0}}_{C_b} \times (\overline{J}_{C_b}^{\phi'_b} )^{\times 3}}\left( \mu_b^{-1}(Z_b)\right) > g-\delta(b)
\]
which by (\ref{codim1}) further implies 
\[
\mathrm{codim}_{(\overline{J}_{C_b}^{\phi'_b} )^{\times 3}}\left( Z_b\right) > g-\delta(b).
\]
Consequently, we deduce from the $\delta$-regularity of $C\to B$ that there is no irreducible codimension~$g$ component contained in the support of $\CK_3$ over $B\setminus B^\circ$. This completes the proof of the proposition.
\end{proof}

\begin{rmk}
    Arinkin's arguments \cite{A2} can further be applied to show that $\CK$ is a Cohen--Macaulay sheaf supported on the Zariski closure of $\Gamma$. This is not needed for our purpose, and we leave it to the interested reader.
\end{rmk}

By Theorem \ref{thm:FM_P} and Proposition \ref{prop:supp}, the data
\[
(\overline{J}^\phi_C, B, \overline{J}^{\phi'}_C, \CP, \CK)
\]
satisfy both (a1, a2) of Definition \ref{def:DAB}. Recall the Fourier transforms $\FF, \FF^{-1}$ induced by $\CP, \CP^{-1}$ respectively. We note that the Fourier vanishing also follows from the $\delta$-regularity of $C\to B$.

\begin{prop}[Fourier vanishing {\cite[Section 3.5]{MSY}}] \label{FV_reduced}
For all $i + j < 2g$, we have
\begin{gather*}
\tag{FV1} \FF^{-1}_j\circ \FF_i = 0 \in \Chow_{3g - i - j + \dim B}(\Jbar_C^{\phi'} \times_B \Jbar_C^{\phi'}), \\
\tag{FV2} \FF_j \circ \FF^{-1}_i = 0 \in \Chow_{3g - i - j + \dim B}(\Jbar_C^{\phi} \times_B \Jbar_C^{\phi}).
\end{gather*}
\end{prop}

\begin{proof}
The proof is identical to that in \cite[Section 3.5]{MSY} for integral curves. The $\delta$-regularity is used here in a parallel way as in the proof of Proposition \ref{prop:supp}. Indeed, instead of $\CK_3$, we consider another sequence of objects
\[
\widetilde{\CK}(N): = \CP^{-1} \circ (i_*\CP)^{\otimes N} \in D^b\mathrm{Coh}(\overline{J}^{\phi'}_C \times \overline{J}^{\phi'}_C), \quad N \in \BZ_{> 0}
\]
where $i: \overline{J}^{\phi'}_C \times_B \overline{J}^{\phi'}_C \hookrightarrow \overline{J}^{\phi'}_C \times \overline{J}^{\phi'}_C$ is the closed embedding. These objects are supported on $\overline{J}^{\phi'}_C \times_B \overline{J}^{\phi'}_C$, and we can deduce the codimension estimate
\[
\mathrm{codim}_{\overline{J}^{\phi'}_C \times_B \overline{J}^{\phi'}_C} \left( \mathrm{Supp}(\widetilde{\CK}(N))\right) \geq g
\]
by a similar proof as for Proposition \ref{prop:supp} using $\delta$-regularity. Then an argument via the Adams operations as in \cite[Section 3.5.2]{MSY} yields the desired (FV1). 

The other half (FV2) can be deduced similarly; see \cite[Proposition 2.9]{MSY2}.
\end{proof}

As before, let $\CH^\phi_{(k)}$ be the shifted perverse sheaves associated with the decomposition of the pushforward complex $\pi_*{\BQ_{\Jbar^\phi_C}}$ as in (\ref{eq:BBD}). The following theorem is an application of Theorem~\ref{thm:DAF} to relative fine compactified Jacobians.

\begin{thm}\label{thm_main_reduced}
    With the notation as above, assume that both morphisms
    \[
    \pi: \Jbar^\phi_C \to B, \quad \pi': \Jbar^{\phi'}_C \to B
    \]
    have full support. We have the following.
    \begin{enumerate}
        \item[(i)] The associated graded 
        \[
        \CH^\phi := \bigoplus_{k = 0}^{2g} \CH^\phi_{(k)}  \in D^b_c(B)
        \]
        is naturally a graded ring object induced by the cup-product on $\pi_*\BQ_{\Jbar^\phi_C}$.
        \item[(ii)] There is an isomorphism
        \[
        \CH^\phi \simeq \CH^{\phi'} \in D^b_c(B)
        \]
        of graded ring objects.
    \end{enumerate}
\end{thm}

\begin{proof}
For (i), we first note that 
\begin{equation}\label{DAB1}
(\overline{J}^\phi_C, B, \overline{J}^{\phi'}_C, \CP_{\phi\phi'}, \CK_{\phi\phi'})
\end{equation}
is a dualizable abelian fibration; here we add subscripts for the Poincar\'e sheaf $\CP$ and the convolution kernel $\CK$ to indicate the dependence on the stability conditions. Indeed, as we mentioned before Proposition \ref{FV_reduced}, (a1, a2) of Definition \ref{def:DAB} are satisfied, and (b) is also satisfied by the assumption of the theorem. Therefore, we obtain (i) from Theorem \ref{thm3.4} and Proposition \ref{FV_reduced}.

Now we prove (ii). By setting $\phi = \phi'$ in (\ref{DAB1}), we obtain another dualizable abelian fibration
\begin{equation}\label{DAB2}
(\overline{J}^{\phi'}_C, B, \overline{J}^{\phi'}_C, \CP_{\phi'\phi'}, \CK_{\phi'\phi'}).
\end{equation}
Proposition \ref{prop:supp} implies that the reduced convolution classes associated with $\CK_{\phi\phi'}$ and $\CK_{\phi'\phi'}$ respectively coincide; both are given by
\[
[\overline{\Gamma}] \in \mathrm{CH}_{2g+ \dim B}(\Jbar^{\phi'}_C\times_B \Jbar^{\phi'}_C\times_B \Jbar^{\phi'}_C). 
\]
Therefore, we obtain (ii) by applying Theorem \ref{thm:DAF} to (\ref{DAB1}) and (\ref{DAB2}).
\end{proof}

\begin{rmk}
    We note that $\delta$-regularity of $C\to B$ is not sufficient to guarantee the full support property (\emph{i.e.}, Definition \ref{def:DAB}(b)) of $\pi: \Jbar^\phi_C \to B$. For example, an elliptic fibration with reducible fibers is always $\delta$-regular, but it has $0$-dimensional supports.  
\end{rmk}

\subsection{Proofs of Theorems \ref{thm0.7} and \ref{thm0.8}}

Theorems \ref{thm0.7} and \ref{thm0.8} are immediate consequences of Theorem \ref{thm_main_reduced}.

\begin{proof}[Proof of Theorem \ref{thm0.7}]
We consider the universal curve
\begin{equation}\label{univ_curve1}
C_{g,n} \to \overline{\CM}_{g,n}
\end{equation}
over the moduli of stable marked curves (or rather, the base change to the moduli of stable marked curves with a level structure constructed in \cite{ACV}). Since we assume $n\geq 1$, there is a section through the smooth locus of the family. Using this section, universal fine compactified Jacobians of different degrees are identified. So we only need to prove that 
\[
\BH^{0,\phi}_{g,n} \simeq \BH^{0,\phi'}_{g,n}
\]
with $\phi, \phi'$ nondegenerate. In view of Theorem \ref{thm_main_reduced}, it suffices to check that 
\begin{enumerate}
    \item[(i)] the universal curve (\ref{univ_curve1}) is $\delta$-regular,
    \item[(ii)] the universal fine compactified Jacobian $\overline{J}^{0,\phi}_{g,n}$ is nonsingular, and
    \item[(iii)] the morphism $\pi: \overline{J}^{0,\phi}_{g,n} \to \overline{\CM}_{g,n}$ has full support,
\end{enumerate}  
where we assume in (ii, iii) that $\phi$ is nondegenerate. 

Indeed,  (ii) is given by \cite[Corollary 4.4]{KP}; (i) and (iii) are given by \cite[Fact 2.4 and Theorem~1.9]{MSV}. This completes the proof.
\end{proof}

\begin{proof}[Proof of Theorem \ref{thm0.8}]
We can use a nonsingular closed point on the reduced locally planar curve $C_0$ to reduce to the case where both $\phi, \phi'$ are nondegenerate stability conditions of degree~$0$ for $C_0$.

Under the assumption of Theorem \ref{thm_main_reduced}, for any closed fiber $C_b$ of the family $C \to B$ there is a multiplicative perverse filtration 
\begin{equation*}    
P_0H^*(\overline{J}^{\phi_b}_{C_b}, \BQ) \subset P_1 H^*(\overline{J}^{\phi_b}_{C_b}, \BQ) \subset \cdots \subset P_{2g}H^*(\overline{J}^{\phi_b}_{C_b}, \BQ) = H^*(\overline{J}^{\phi_b}_{C_b}, \BQ),
\end{equation*}    
so that the associated graded $\BH_{C_b}^{\phi_b}$ is a bigraded $\BQ$-algebra; for two nondegenerate stability conditions $\phi, \phi'$, we further have an isomorphism of bigraded $\BQ$-algebras:
\[
\BH^{\phi_b}_{C_b} \simeq \BH^{\phi'_b}_{C_b}.
\]
Therefore, to prove Theorem \ref{thm0.8}, it suffices to check that for any reduced locally planar curve~$C_0$, there is a $\delta$-regular family $C\to B$ containing $C_0$ as a closed fiber, such that
\begin{enumerate}
    \item[(i)] any nondegenerate stability condition of degree $0$ for $C_0$ can be extended to a nondegenerate stability condition for $C\to B$,
    \item[(ii)] the relative fine compactified Jacobian associated with the nondegenerate stability condition given by (i) is nonsingular, and the morphism to the base has full support.
\end{enumerate}

Indeed, both (i) and (ii) are guaranteed by taking a versal deformation $C\to B$ of $C_0$ as in~\cite[Facts 2.3, 2.4, and Theorem 2.12]{MSV}, and the proof is complete.
\end{proof}

\section{Dependence on the stability}\label{sec:4}
In this section we complete the proof of Theorem \ref{thm0.6}. Discussions in Sections \ref{sec:4.1}-\ref{sec:new4.3} apply generally to arbitrary $g$ and $n\geq 1$. Starting from the end of Section \ref{sec:4.3} we specialize to the case $n=1$.

\subsection{Space of stability conditions} \label{sec:4.1}
For the definition of stability conditions, we refer to Kass--Pagani \cite{KP}. Recall that a stability condition $\phi$ is nondegenerate if there exists no strictly~$\phi$-semistable multidegree.

Let $\Mtl_{g,n}\subset\Mbar_{g,n}$ be the locus of treelike curves; these are stable marked curves whose graph is a tree with any number of self-loops attached, and they form an open substack of~$\Mbar_{g,n}$. A stability condition $\phi$ of degree $0$ is called {\em semismall} if the trivial line bundle on the universal curve over~$\Mtl_{g,n}$ is $\phi$-stable. When $\phi$ is semismall and nondegenerate, $\phi$-stable fine compactified Jacobians are canonically isomorphic over $\Mtl_{g,n}$.

Since we assumed $n \geq 1$, for any nondegenerate stability condition $\phi$ of degree $d$ we can find $\phi'$ nondegenerate and semismall such that the corresponding universal fine compactified Jacobians are isomorphic to each other:
\[
\Jbar^{d, \phi}_{g,n} \simeq \Jbar^{0, \phi'}_{g,n}.
\]
To see this, we can use the section given by the marking to change the degree, and twist by a vertical divisor to change the bidegree for curves with $2$ components and $1$ separating node; this governs the locus of treelike curves $\overline{\CM}^{\mathrm{tl}}_{g,n} \subset \overline{\CM}_{g,n}$ by \cite[Corollary 3.6]{KP}.

\subsection{Divisors on fine compactified Jacobians}\label{sec:4.2}
From now on, we assume that $\phi$ is a nondegenerate stability condition of degree $d$, and $\pi: \Jbar^{d,\phi}_{g,n}\to \overline{\CM}_{g,n}$ is the corresponding universal fine compactified Jacobian.

Let $\mathcal{C}_{g,n} \to \Jbar_{g,n}^{d,\phi}$ be the universal curve and let $\mathcal{F}$ be the universal sheaf on $\mathcal{C}_{g,n}$ trivialized along the section given by the first marking. By \cite{EP}, the universal sheaf $\mathcal{F}$ is the pushforward of the universal admissible line bundle $\mathcal{L}$ on the universal quasistable curve
\[
p : \mathcal{C}_{g,n}^\qs \to \Jbar_{g,n}^{d,\phi}.
\]
Here the admissible line bundle has degree $1$ on each unstable component of $\mathcal{C}_{g,n}^\qs$.
This universal family induces natural tautological classes, which generate a ring
\[R^*(\Jbar_{g,n}^{d,\phi})\subset \Chow^*(\Jbar_{g,n}^{d,\phi}),\]
called the \emph{tautological ring} of the universal fine compactified Jacobian $\Jbar^{d,\phi}_{g,n}$. Applying the cycle class map, we can also consider the cohomological counter-part of the tautological ring  
\[
RH^*(\Jbar^{d,\phi}_{g,n}, \BQ)\subset H^*(\Jbar^{d,\phi}_{g,n},\BQ).
\]
Although most of the cohomological arguments in this section take place in the tautological ring $RH^*(\Jbar^{d,\phi}_{g,n}, \BQ)$,  all relevant tautological classes will be given in explicit form. For this reason, we do not provide a complete definition of the tautological ring here, and instead refer the interested reader to \cite{BHPSS, BMP} for further details.

We first introduce several (tautological) divisor classes on $\Jbar_{g,n}^{d,\phi}$; they will be used throughout this section:
\[\Theta:=-\frac{1}{2}p_*(c_1(\mathcal{L})^2), \quad  \kappa_{0,1}:=p_*(c_1(\mathcal{L}) \cup c_1(\omega_{p,\log})),\]
and
\[\xi_i=x_i^*c_1(\mathcal{L}), \quad \psi_i = x_i^*c_1(\omega_p), \quad 1\le i\le n.\]
Here $x_i: \Jbar_{g,n}^{d,\phi} \to \mathcal{C}_{g,n}^\qs$ is the section corresponding to the $i$-th marking, and $\omega_p$ (resp.~$\omega_{p,\log}$) is the relative dualizing (resp.~log-canonical) line bundle.

We also recall the tautological ring 
\[
RH^*(\Mbar_{g,n}, \BQ) \subset H^*(\Mbar_{g,n}, \BQ)
\]
and the standard tautological classes
\[
\kappa_i, \lambda_i \in RH^{2i}(\overline{\CM}_{g,n}, \BQ), \quad  \psi_i \in RH^2(\overline{\CM}_{g,n}, \BQ);
\]
see \cite{Pand}. Classes on $\Mbar_{g,n}$ can also be viewed as classes on $\Jbar^{d,\phi}_{g,n}$ via the pullback
\[
\pi^*: H^*(\Mbar_{g,n}, \BQ) \hookrightarrow H^*(\Jbar^{d,\phi}_{g,n}, \BQ).
\]

\begin{lem}\label{lem:divisor}
    We have
    \[RH^2(\Jbar^{d,\phi}_{g,n},\BQ)=H^2(\Jbar_{g,n}^{d,\phi},\BQ).
    \]
    In particular, the group $H^2(\Jbar_{g,1}^{d,\phi},\BQ)$ is spanned by the classes $\Theta, \kappa_{0,1}$ modulo classes pulled back from $RH^2(\Mbar_{g,1},\BQ)$. 
\end{lem}
\begin{proof}
    By \cite{FV}, the Chow group $\Chow^1(J_{g,n})$ is spanned by tautological classes. Since $\Jbar_{g,n}^{d,\phi}$ is nonsingular, the excision sequence implies that $R^1(\Jbar_{g,n}^{d,\phi}) = \Chow^1(\Jbar_{g,n}^{d,\phi})$. By the Lefschetz $(1,1)$ theorem, it suffices to show that $H^{2,0}(\Jbar_{g,n}^{d,\phi}) = 0$.
    
    The rational Abel--Jacobi map $\Mbar_{g,n+g}\dashrightarrow \Jbar_{g,n}^{d,\phi}$ with respect to the last $g$ markings admits a resolution $a: \Mbar_{g,n+g}'\to\Jbar_{g,n}^{d,\phi}$ via a birational modification $\Mbar_{g,n+g}' \to \Mbar_{g,n+g}$ with~$\Mbar_{g,n+g}'$ nonsingular. Since $a$ is proper and surjective, the pullback map
    \[
    a^*:H^*(\Jbar^{d, \phi}_{g,n},\BQ)\to H^*(\Mbar_{g,n+g}',\BQ)
    \]
    is injective. As birational transformations do not affect $H^{2, 0}(-)$, the vanishing of~$H^{2, 0}(\Jbar^{d,\phi}_{g,n})$ follows from the fact that $H^{2,0}(\Mbar_{g,n+g}) = 0$ by \cite[Theorem 2.2]{AC}.
\end{proof}

\subsection{Vanishing via the perverse filtration} \label{sec:new4.3}

Now we consider the perverse filtration \eqref{p-filtration} associated with the fibration $\pi : \Jbar^{d,\phi}_{g,n} \to \Mbar_{g,n}$. We recall the following proposition from \cite[Proposition 8.7]{BMP}.

\begin{prop}\label{pro:kappa}
Taking cup-product with the class $\kappa_{0,1}\in H^2(\Jbar_{g,n}^{d,\phi},\BQ)$ satisfies
\[
\kappa_{0,1} \cup : P_kH^m(\Jbar_{g,n}^{d,\phi}, \BQ)  \to P_{k+1}H^{m+2}(\Jbar_{g,n}^{d,\phi}, \BQ).
\]
for any $k, m$. The same holds for the class $\xi_i \in H^2(\Jbar^{d,\phi}_{g,n},\BQ)$.
\end{prop}

This means that the class $\kappa_{0,1}$ and $\xi_i$ have \emph{strong perversity} $1$ in the sense of \cite{MS_PW}. For example, Proposition \ref{pro:kappa} immediately implies
\begin{equation}\label{ThetaKappa}
\Theta^k \cup \kappa^l_{0,1} \cup \prod_i \xi^{m_i}_i \in P_{2k+l+\sum_i m_i} H^{2k + 2l +  2\sum_i{m_i}}(\Jbar_{g,n}^{d,\phi},\BQ).
\end{equation}

We also recall the following lemma (see \cite[Lemma 8.6]{BMP}) which shows that perversity implies vanishing after pushing forward.

\begin{lem}\label{lem:lowp}
    If $\alpha\in P_{2g-1}H^{*}(\Jbar^{d,\phi}_{g,n},\BQ)$, then
    \[
    \pi_*\alpha=0 \in H^*(\Mbar_{g, n}, \BQ).
    \]
\end{lem}

As we will apply later, Proposition \ref{pro:kappa} and Lemma \ref{lem:lowp} yield vanishing of the classes 
\[
\pi_*( \Theta^k \cup  \kappa_{0,1}^l )
\]
for many values of $k,l$; this significantly simplifies our calculations.

\subsection{Outline of the proof} \label{sec:new4.4}

We prove Theorem~\ref{thm0.6} by contradiction. Since $\Jbar^{d,\phi}_{g,1}$ is isomorphic to $\Jbar_{g,1}^{0,\phi'}$ for some nondegenerate semismall stability condition $\phi'$, we only consider stability conditions which are nondegenerate and semismall.

Let $\phi_1, \phi_2$ be nondegenerate semismall stability conditions, and suppose that there exists a graded ring isomorphism
\begin{equation}\label{eq:isom}
f : H^*(\Jbar_{g,1}^{0,\phi_1}, \BQ) \xrightarrow{\simeq} H^*(\Jbar_{g,1}^{0,\phi_2}, \BQ)
\end{equation}
which is linear over $H^*(\Mbar_{g,1}, \BQ)$. Since the top-degree cohomology of both fine compactified Jacobians is one-dimensional, there exists a nonzero constant $\cc \in \BQ^\times$ depending only on $f$ such~that
\begin{equation*}\label{eq:e}
    \cc \cdot \int_{\Jbar_{g,1}^{0,\phi_1}} \Xi =  \int_{\Jbar_{g,1}^{0,\phi_2}}f(\Xi)
\end{equation*}
for all $\Xi\in H^*(\Jbar_{g,1}^{0,\phi_1},\BQ)$. In particular, for any $\gamma \in H^*(\Mbar_{g,1}, \BQ)$ and $k,l \in \BZ_{\geq 0}$, we must have 
\begin{equation*}
    \cc \cdot \int_{\Jbar_{g,1}^{0,\phi_1}} \Theta^k\cup\kappa_{0,1}^l \cup \gamma =  \int_{\Jbar_{g,1}^{0,\phi_2}} f(\Theta)^k\cup f(\kappa_{0,1})^l\cup\gamma.
    \end{equation*}
For $i = 1,2$, let $\pi_i : \Jbar_{g,1}^{0,\phi_i}\to \Mbar_{g,1}$ be the natural projection. By the Poincar\'e duality for~$\Mbar_{g,1}$, the equality of the integrals above implies 
\begin{equation}\label{eq:int}
    \cc \cdot \pi_{1*}(\Theta^k \cup \kappa_{0,1}^l) = \pi_{2*}(f(\Theta)^k \cup f(\kappa_{0,1})^l) \in H^*(\Mbar_{g,1},\BQ).
\end{equation}
For our purpose, it suffices to show that \eqref{eq:int} does not hold for some choices of nondegenerate semismall stability conditions $\phi_1, \phi_2$.

By Lemma \ref{lem:divisor}, there exist unique 
\[
a, b, s, t \in \BQ, \quad \beta, \beta' \in RH^2(\Mbar_{g,1}, \BQ)
\]
such that
\begin{equation*}\label{eq:ab}
f(\Theta) = a\Theta + b\kappa_{0,1} + \beta, \quad f(\kappa_{0,1}) = s\Theta + t\kappa_{0,1} + \beta'.
\end{equation*}
As $k$ and $l$ varies, the relation \eqref{eq:int} imposes constraints on the coefficients $\cc, a, b, s, t$, and the classes $\beta, \beta'$. In the following, we show that these constraints cannot always be satisfied as long as the genus $g$ is at least~$4$. Our main tool is the tautological relations obtained from the theory of universal double ramification cycles \cite{BHPSS} over the universal Picard stack. These relations, as we will review in the next section, impose strong constraints on the classes 
\[
\pi_{i*}(\Theta^k \cup \kappa_{0,1}^l) \in H^*(\Mbar_{g,1}, \BQ).
\]

\subsection{Universal double ramification cycle relations} \label{sec:4.3}

Let $\pic_{g,n}$ be the universal Picard stack over the moduli stack $\FM_{g,n}$ of prestable curves of genus $g$ with $n$ markings. Let $\pic_{g,n,d}$ be the connected component parameterizing line bundles of total degree $d$.

Let $\phi$ be a nondegenerate stability condition of degree $d$. Using the universal admissible line bundle as in Section \ref{sec:4.2}, there exists a morphism to the universal Picard stack
\begin{equation}\label{eq:varphi}
\varphi: \Jbar^{d,\phi}_{g,n}\to \pic_{g,n,d}.
\end{equation}
The tautological divisor classes in Section \ref{sec:4.2} induced by the universal admissible line bundle~$\CL$ are naturally defined on the universal Picard stack, whose pullback along $\varphi$ recover the corresponding classes on the universal fine compactified Jacobians; we will use the same notation to denote these classes on $\pic_{g,n,d}$ and $\Jbar_{g,n}^{d,\phi}$ respectively. The universal double ramification cycle relations are on $\pic_{g,n,d}$; pulling back along \eqref{eq:varphi} yields relations on $\Jbar_{g,n}^{d,\phi}$ for any nondegenerate stability condition $\phi$.

We first recall the universal double ramification cycle introduced in \cite{BHPSS} (see also \cite{BMP}). Let $d$ be a fixed integer. Consider a tuple of integers $\sfa =(a_1,\ldots, a_n) \in \BZ^n$ and $b \in \BZ$ with
\[
\sum_{i} a_i = d + b(2g - 2 + n).
\]
For given integers $\sfa$ and $b$, the \emph{universal double ramification cycle class}
\begin{equation*}
\uniDR_g^c(b;\sfa) \in \Chow^c(\pic_{g,n,d})
\end{equation*}
is a certain algebraic class that governs the Abel--Jacobi theory of families of prestable curves.

We state in the following theorem some properties of the universal double ramification cycles needed in our argument. For $N \in \BZ$, the ``multiplication by $N$'' map
\[[N]:\pic_{g,n,0}\to\pic_{g,n,0}\]
is given by the $N$-th tensor power of the universal line bundle.
\begin{thm}\label{thm:uniDR}
    Let $d \in \BZ$ be fixed. Let $\sfa\in\BZ^n$ and $b\in\BZ$ with $\sum_i a_i = d + b(2g-2+n)$.
    \begin{enumerate}
        \item[(i)] When $d=0$, the pullback
        \[
        [N]^*\uniDR_g^c(b;\sfa)
        \]
        admits an explicit expression as a polynomial in $N$ of degree $2c$.
        \item[(ii)] The class $\uniDR^c_g(b;\sfa)$ admits an explicit expression as a polynomial in $a_2, \ldots, a_n$, and~$b$ of total degree $2c$. Here the variable $a_1$ does not show up due to the relation between the $a_i$ and $b$.
        \item[(iii)] For $c>g$, we have the vanishing
        \[
        \uniDR^c_g(b;\sfa)=0 \in \Chow^c(\pic_{g,n,d}).
        \]
    \end{enumerate}
\end{thm}
\begin{proof}
     Part (i) is from \cite[Proposition 6.4]{BMP}. When $d=0$, (ii) is from \cite{Pixton, Sp}. For general $d \in \BZ$, we consider the isomorphism given by twisting by the first marking
    \[
    \tau_d : \pic_{g,n,d} \xrightarrow{\simeq} \pic_{g,n,0}, \quad (C,\{x_i\}, L) \to (C, \{ x_i\}, L(-dx_1)).
    \]
    Then (ii) follows from the functoriality of $\uniDR^c_g(b;\sfa)$ under $\tau_d^*$; see \cite[Section 7.4]{BHPSS}. Part (iii) is from \cite[Theorem 0.8]{BHPSS}.
\end{proof}

We mainly focus on the case $d=0, n = 1$, for which we can refine the relation in Theorem~\ref{thm:uniDR}(iii) as follows. For all $w, m \geq 0$, we have
\begin{equation}\label{eq:refine}
[\uniDR_g^c(b; a_1)]_{\text{weight} = w,\, \deg = m} = 0,
\end{equation}
where the weight $w$ refers to the coefficient of $N^w$, and the degree $m$ is the degree as a polynomial in $b$. The left-hand side of (\ref{eq:refine}) admits an explicit expression in terms of tautological classes. We note that a special case of \eqref{eq:refine} was used in \cite{BMSY}.

We present the explicit forms of some special cases of \eqref{eq:refine}, which will be used in our proof; we refer to \cite[Section~0.3]{BHPSS} for further details. We consider pairs $(\Gamma, \delta)$ consisting of a prestable graph~$\Gamma$ of genus $g$ with $n$ legs, and a degree $\delta: V(\Gamma) \to \BZ$ on the set of vertices of $\Gamma$. Each pair defines a Picard substack $\pic_{\Gamma_\delta}$ with prescribed degeneration, together with a natural morphism 
\[
j_{\Gamma_\delta}: \pic_{\Gamma_\delta} \to \pic_{g,n}
\]
of degree $|\mathrm{Aut}(\Gamma_\delta)|$. We use $[\Gamma_\delta]$ to denote the class given by the pushforward of the fundamental class $(j_{\Gamma_\delta})_*1$. For a fine compactified Jacobian $\Jbar^{d,\phi}_{g,n}$ with a $\phi$-stable multidegree $\delta$ on a quasistable graph $\Gamma$, let $\Jbar^{d,\phi}_{\Gamma_\delta}$ denote the pullback of $\pic_{\Gamma_\delta}$ along \eqref{eq:varphi}. 

Partial normalization along the edges of $\Gamma$ induces a morphism
\begin{equation}\label{eq:pn}
    \pic_{\Gamma_\delta}\to\prod_{v\in V(\Gamma)}\pic_{g(v),n(v),\delta(v)}\,.
\end{equation}
We also need to consider generalizations of the classes $[\Gamma_\delta]$; they are natural classes of the form
\[
[\Gamma_\delta, \alpha]
\]
given by the stratum $\Gamma_\delta$ decorated by a tautological class $\alpha$ pulled back from the factors $\pic_{g(v),n(v),\delta(v)}$. These generalized boundary classes appear as terms of the double ramification cycle relations.

As an example of (\ref{eq:refine}), the relation 
\[
[\uniDR^{g+1}_g(b; a_1)]_{\mathrm{weight}=2g+1,\,\deg=1}=0
\]
on the Picard stack $\pic_{g,1,0}$ is of the form
\begin{equation}\label{eq:rel2}
\begin{aligned}
    \frac{\Theta^{g}\cup\kappa_{0,1}}{g!} ={} & {- \frac{\Theta^g}{g!}} \cup \sum_{g_1+g_2 = g} \sum_{d+d'=0} d(2g_1-1) \Big[\begin{tikzpicture}[baseline=-0.1 cm, vertex/.style={circle,draw,font=\Large,scale=0.5, thick}]
    \node[vertex] (A) at (-1.5,0) [label=below:$g_2$] {$d'$};
    \node[vertex] (B) at (0,0) [label=below:$g_1$] {$d$};
    \draw[thick] (A) to (B);
    \draw[thick] (A) to (-2, 0) node[left] {$1$};
    \end{tikzpicture}\Big]\\
    &+\frac{\Theta^{g-1}}{(g-1)!}\cup \sum_{g_1+g_2=g-1} \sum_{d+d'=0}\frac{d^3g_1}{6} \Big[\begin{tikzpicture}[baseline=-0.1 cm, vertex/.style={circle,draw,font=\Large,scale=0.5, thick}]
    \node[vertex] (A) at (-1.5,0) [label=below:$g_2$] {$d'$};
    \node[vertex] (B) at (0,0) [label=below:$g_1$] {$d$};
    \draw[thick] (A) to[bend left] (B);
    \draw[thick] (A) to[bend right] (B);
    \draw[thick] (A) to (-2, 0) node[left] {$1$};
    \end{tikzpicture}\Big]\\ 
    &+\text{(other terms)}.
\end{aligned}
\end{equation}
Here the number $\delta(v)$ at a vertex $v$ in the graphs above refers to the assigned degree, and the number $g_i$ below the vertex is the genus of the vertex. The infinite sum becomes finite after pulling back the relation along \eqref{eq:varphi}. 


\begin{rmk}
    In the formula (\ref{eq:refine}), we only explicitly wrote out the
    main terms of the identity.  Every term of the identity is of the form
    \[
    \mathrm{coefficient} \cdot \Theta^{k} \cup [\Gamma_\delta, \alpha].
    \]
    As we discussed in Section \ref{sec:4.2}, our main application of the universal double ramification cycle relations is to handle the pushforward of monomials of $\Theta$ and $\kappa_{0,1}$ to $\Mbar_{g,1}$. After pulling back the relation via (\ref{eq:varphi}) and pushing forward to $\Mbar_{g,1}$, many terms vanish for dimension reasons.  Namely, if the stabilization of the graph $\Gamma$ has higher codimension than the pushforward class, then the corresponding term will vanish. In particular, these terms are not relevant for our arguments so we use \emph{other terms} to denote them for notational convenience. This convention is applied throughout Section \ref{sec:4} for all universal double ramification cycle relations.
\end{rmk}


We recall some further relations on the stacks
\[
\pic_{g,1,0}, \quad \pic_{g,2,d}, \quad \pic_{g,3,d}
\]
given by Theorem~\ref{thm:uniDR}(iii); although we mainly consider the case $d = 0, n = 1$ as in (\ref{eq:rel2}), we need the other two stacks to handle boundary strata.

\medskip
{\bf \noindent $\bullet$ On $\pic_{g,1,0}$:} we first define
\[
\widetilde{\kappa}_1 := \frac{1}{2}\left( {-\kappa_1} + (2g-1)^2\psi_1 - \sum_{\substack{g_1 + g_2 = g \\ g_1 , g_2\ge 1}} (2g_1 - 1)^2 \Big[\begin{tikzpicture}[baseline=-0.1 cm, vertex/.style={circle,draw,font=\Large,scale=0.5, thick}]
    \node[vertex] (A) at (-1.5,0) [label=below:$g_2$] {};
    \node[vertex] (B) at (0,0) [label=below:$g_1$] {};
    \draw[thick] (A) to (B);
    \draw[thick] (A) to (-2, 0) node[left] {$1$};
    \end{tikzpicture}\Big] \right) \in \Chow^1(\mathfrak{M}_{g,1}).
    \]
    Here $\mathfrak{M}_{g,1}$ is the moduli stack of prestable curves with $1$ marking, and $\kappa_1,\psi_1$ and the boundary strata (appeared on the right-hand side) are standard tautological classes on $\mathfrak{M}_{g,1}$ as in \cite{BS1}. The relation
\[
[\uniDR^{g+1}_g(b;a_1)]_{\mathrm{weight}=2g,\,\deg =2} = 0
\]
is of the form 
\begin{equation}\label{eq:rel22}
    \begin{aligned}
    \frac{\Theta^{g-1} \cup \kappa_{0,1}^2}{(g-1)!2!} ={} & {-\frac{\Theta^g\cup\widetilde{\kappa}_1}{g!}}\\
    &+ \frac{\Theta^{g-1}}{(g-1)!} \cup \sum_{g_1+g_2=g-1} \sum_{d+d'=0}\frac{6d^2(2g_1)^2}{48} \Big[\begin{tikzpicture}[baseline=-0.1 cm, vertex/.style={circle,draw,font=\Large,scale=0.5, thick}]
    \node[vertex] (A) at (-1.5,0) [label=below:$g_2$] {$d'$};
    \node[vertex] (B) at (0,0) [label=below:$g_1$] {$d$};
    \draw[thick] (A) to[bend left] (B);
    \draw[thick] (A) to[bend right] (B);
    \draw[thick] (A) to (-2, 0) node[left] {$1$};
    \end{tikzpicture}\Big]\\
    &+ \frac{\Theta^{g-2}\cup\kappa_{0,1}}{(g-2)!} \cup \sum_{g_1+g_2=g-1} \sum_{d+d'=0}\frac{d^3(2g_1)}{12} \Big[\begin{tikzpicture}[baseline=-0.1 cm, vertex/.style={circle,draw,font=\Large,scale=0.5, thick}]
    \node[vertex] (A) at (-1.5,0) [label=below:$g_2$] {$d'$};
    \node[vertex] (B) at (0,0) [label=below:$g_1$] {$d$};
    \draw[thick] (A) to[bend left] (B);
    \draw[thick] (A) to[bend right] (B);
    \draw[thick] (A) to (-2, 0) node[left] {$1$};
    \end{tikzpicture}\Big]\\
    & + \frac{\Theta^{g-2}}{(g-2)!}\cup \sum_{g_1+g_2=g-1} \sum_{d+d'=0}\frac{d^4(2g_1)^2}{32} \Big[\begin{tikzpicture}[baseline=-0.1 cm, vertex/.style={circle,draw,font=\Large,scale=0.5, thick}]
    \node[vertex] (A) at (-1.5,0) [label=below:$g_2$] {$d'$};
    \node[vertex] (B) at (0,0) [label=below:$g_1$] {$d$};
    \draw[thick] (A) to[bend left] node[midway, above]{$\psi_h+\psi_{h'}$}  (B);
    \draw[thick] (A) to[bend right] (B);
    \draw[thick] (A) to (-2, 0) node[left] {$1$};
    \end{tikzpicture}\Big] \\
    &+ \text{(other terms)}.
\end{aligned}    
\end{equation}
Here in the fourth term on the right-hand side, we use $e = (h, h')$ to denote an edge, and $\psi_h, \psi_{h'}$ are the $\psi$-classes at the corresponding half-edges.

\medskip
{\bf \noindent $\bullet$ On $\pic_{g,2,d}$:} the relation
\[
[\uniDR^{g+1}_g (b; 0, 2gb + d)]_{\mathrm{deg} = 1} = 0
\]
is of the form
\begin{equation}\label{eq:rel3}
\begin{aligned}
    \frac{\Theta^{g}\cup \kappa_{0,1}}{g!} ={}& 2g\cdot\frac{\Theta^g\cup \xi_2}{g!} \\
    &+ 2gd\cdot \frac{\Theta^g\cup \psi_2}{g!} \\
    &+ \frac{\Theta^g}{g!} \cup \sum_{g_1+g_2=g} \sum_{e+e' = d} -2g_1e \cdot \Big[\begin{tikzpicture}[baseline=-0.1 cm, vertex/.style={circle,draw,font=\Large,scale=0.5, thick}]
    \node[vertex] (A) at (-1.5,0) [label=below:$g_1$] {$e$};
    \node[vertex] (B) at (0,0) [label=below:$g_2$] {$e'$};
    \draw[thick] (A) to (B);
    \draw[thick] (A) to (-2, 0) node[left] {$1$};
    \draw[thick] (B) to (0.5,0) node[right]{$2$};
    \end{tikzpicture}\Big]\\
    & + \text{(other terms)}.
\end{aligned}
\end{equation}

\medskip
{\bf \noindent $\bullet$ On $\pic_{g,3,d}$:} the relation  
\[
[\uniDR_g^{g+1}(0;a,d-a,0)]_{\mathrm{deg} = 1} = 0
\]
is of the form
\begin{equation}\label{eq:rel4}
\begin{aligned}
    \frac{\Theta^g \cup (\xi_2- \xi_1)}{g!} ={} & \frac{\Theta^g}{g!}\cup \sum_{g_1+g_2=g} \sum_{e+e'=d} e \cdot \Big[\begin{tikzpicture}[baseline=-0.1 cm, vertex/.style={circle,draw,font=\Large,scale=0.5, thick}]
    \node[vertex] (A) at (-1.5,0) [label=below:$g_1$] {$e$};
    \node[vertex] (B) at (0,0) [label=below:$g_2$] {$e'$};
    \draw[thick] (A) to (B);
    \draw[thick] (A) to (-2, 0) node[left] {$1$};
    \draw[thick] (B) to (0.5,0) node[right]{$2$};
    \draw[thick] (A) to (-1.5,0.5) node[left] {$3$};
    \end{tikzpicture}\Big]\\
    &+ \frac{\Theta^g}{g!} \cup \sum_{g_1+g_2=g} \sum_{e+e'=d} (d-e) \cdot \Big[\begin{tikzpicture}[baseline=-0.1 cm, vertex/.style={circle,draw,font=\Large,scale=0.5, thick}]
    \node[vertex] (A) at (-1.5,0) [label=below:$g_1$] {$e$};
    \node[vertex] (B) at (0,0) [label=below:$g_2$] {$e'$};
    \draw[thick] (A) to (B);
    \draw[thick] (A) to (-2, 0) node[left] {$2$};
    \draw[thick] (B) to (0.5,0) node[right]{$1$};
    \draw[thick] (A) to (-1.5,0.5) node[left] {$3$};
    \end{tikzpicture}\Big]\\
    &+ \text{(other terms)},
\end{aligned}    
\end{equation}
and the relation
\[
[\uniDR^{g+1}_g(0;a,d-a,0)]_{\mathrm{deg}= 0} = 0
\]
is of the form
\begin{equation}\label{eq:rel5}
    \frac{\Theta^{g+1}}{(g+1)!} = -\frac{\Theta^g}{g!}\cup \left(d \xi_2+\frac{d^2}{2}\psi_2\right) +\text{(other terms)}.
\end{equation}

\medskip

We conclude this section by recalling a result of \cite{BMP} where a closed formula for the pushforward of monomials of divisors is obtained when the perversities of the divisors sum exactly to $2g$; these classes are presented in terms of the double ramification cycle formulas over the moduli of stable curves $\Mbar_{g,n}$ \cite{JPPZ}. Let $\sfa \in \BZ^n$ and $b \in \BZ$ with $\sum_i a_i = b(2g-2+n)$. The trivial line bundle induces a map $e : \Mbar_{g,n} \to \mathfrak{Pic}_{g,n,0}$. The double ramification cycle formula is then defined by
\[
\dr^c_g(b;\sfa) = e^* \uniDR_g^c(b;\sfa).
\]
By Theorem \ref{thm:uniDR}(ii), this is a polynomial in $a_2, \ldots, a_n$, and $b$; we use
\[
[\dr^c_g(b;\sfa)]_{b^m a_2^{k_2}\cdots a_n^{k_n}}
\]
to denote the coefficient of the monomial $b^m a_2^{k_2} \cdots a_n^{k_n}$ of $\dr^c_g(b ;\sfa)$. We refer to \cite{JPPZ} for explicit formulas for these classes.

\begin{thm}[{\cite[Theorem 1.1]{BMP}}]\label{thm:bmp} 
Let $\pi : \Jbar_{g,n}^{d,\phi}\to \Mbar_{g,n}$ be the universal fine compactified Jacobian associated with a nondegenerate stability condition $\phi$ where the universal line bundle is trivialized along the first marking. Let $l,m, k_2, \cdots, k_n$ be nonnegative integers with
\[
2l + m + \sum_{i=2}^n k_i = 2g.\]
Then we have 
\[
\pi_*\left( \frac{\Theta^l}{l !} \cup \frac{(-\kappa_{0,1})^m}{m!} \cup \prod_{i=2}^n\frac{\xi^{k_i}}{k_i!} \right) = (-1)^{g-l} [\dr^c_g(b;\sfa)]_{b^m a_2^{k_2}\cdots a_n^{k_n}};
\]
in particular, the pushforward is independent of the stability condition $\phi$.    
\end{thm}

\subsection{Constraints}

In this section, we assume the existence of the ring isomorphism (\ref{eq:isom}) in Section \ref{sec:new4.4}, so that the relation \eqref{eq:int} holds for $\phi_1, \phi_2$. We derive constraints on $\cc, a, b, s, t \in \BQ$, and $\beta, \beta' \in RH^2(\Mbar_{g,1}, \BQ)$.

\begin{lem}\label{lem:a}
   We have 
    \[
    a^g = \cc \neq 0,\quad  s = 0.
    \]
\end{lem}
\begin{proof}
    Since $\pi_{i*}(\Theta^g)=g!$, we have 
    \[
    \cc \cdot g!=\pi_{1*}(\Theta^g)=\pi_{2*}((a\Theta+b\kappa_{0,1}+\beta)^g)=a^gg!
    \]
    where the second equality follows from (\ref{eq:int}) and the last equality follows from the perversity bound of Proposition \ref{pro:kappa} and the vanishing given by Lemma \ref{lem:lowp}. We conclude that $a^g = \cc$.

    By a similar argument using the perversity bound of Proposition \ref{pro:kappa}, we have
    \[
    0 = \cc \cdot \pi_{1*}( \kappa_{0,1}^g ) = \pi_{2*}( (s\Theta + t\kappa_{0,1} + \beta')^g) = s^g g!.
    \]
    Therefore $s=0$.
\end{proof}


\begin{lem}\label{lem:v}
    If $g\ge 2$, we have
    \[
    a = t^2.
    \]
\end{lem}
\begin{proof}
    We first prove that $[\dr^1_g(b;a_1)]_{b^2}$ is a nonzero class in $H^2(\Mbar_{g,1},\BQ)$. Restricting to $\CM_{g,1}\subset \Mbar_{g,1}$, we have
    \begin{equation*}\label{eq:dr11}
        [\dr^1_g(b;a_1)]_{b^2} = \frac{1}{2}(-\kappa_1 + (2g-1)^2\psi_1).
    \end{equation*}
    Let $p : \CM_{g,1} \to \CM_g$ be the forgetful map. Then we have 
    \[
    p_*(-\kappa_1 + (2g-1)^2 \psi_1) = 8g(g-1)^2 \cdot [\CM_{g}] \neq 0 \in H^0(\CM_g ,\BQ).
    \]
Therefore, we deduce from Theorem \ref{thm:bmp} that
    \begin{equation} \label{eq:dr1}
        \pi_{1*}\big( \Theta^{g-1} \cup \kappa_{0,1}^2 \big) = \pi_{2*}( \Theta^{g-1} \cup \kappa_{0,1}^2 )  \neq 0.
    \end{equation}

     On the other hand, by \eqref{eq:int} combined with \eqref{ThetaKappa}, we obtain that
    \[
    \cc \cdot \pi_{1*}(\Theta^{g-1} \cup \kappa_{0,1}^2) = \pi_{2*}( (a\Theta + b\kappa_{0,1} + \beta)^{g-1} \cup (t \kappa_{0,1} + \beta')^2) = a^{g-1} t^2 \cdot  \pi_{2*}(\Theta^{g-1}\cup \kappa_{0,1}^2),
    \]
    from which we deduce
    \[
    a^g = \cc  = a^{g-1}t^2.
    \]
Here we applied Lemma \ref{lem:a} and \eqref{eq:dr1}. The desired identity is concluded.
\end{proof}

\begin{lem}\label{lem:beta}
    If $g\ge 2$, we have
    \[
    \beta' = \frac{2b}{t}[\dr^1_g(b;a_1)]_{b^2}.
    \]
\end{lem}
\begin{proof}
We first prove that 
\begin{equation}\label{eq:kappa1}
        \pi_{1*}(\Theta^g\cup \kappa_{0,1}) = {\pi_{2*}}( \Theta^g\cup \kappa_{0,1} ) = 0.
    \end{equation}
 We pull back the relation (\ref{eq:rel2}) to $\Jbar_{g,1}^{0,\phi_i}$ and then push it forward to $\Mbar_{g,1}$. The left-hand side calculates (\ref{eq:kappa1}). Now we consider the right-hand side:  The first term does not contribute since the stability condition forces $d = d' = 0$. By dimension considerations, the only potentially contributing term from the second term is when $g_1 = 0$; however, this term also vanishes because its coefficient is a multiple of $g_1$. Therefore, \eqref{eq:kappa1} is concluded.
    
    By (\ref{eq:int}), we obtain
    \[
    0 = \cc \cdot \pi_{1*}(\Theta^g \cup \kappa_{0,1}) = \pi_{2*} ((a\Theta + b\kappa_{0,1} +\beta)^g \cup (t\kappa_{0,1} + \beta')).
    \]
    Expanding the right-hand side, only the terms
    \begin{equation}\label{terms}
        \pi_{2*} (\Theta^g \cup \kappa_{0,1}), \quad  \pi_{2*}(\Theta^{g-1} \cup \kappa_{0,1}^2), \quad \pi_{2*}(\Theta^g\cup \beta')
    \end{equation}
    contribute by (\ref{ThetaKappa}) and Lemma \ref{lem:lowp}. We conclude the desired formula for $\beta'$ by comparing the coefficients of the terms in \eqref{terms} in the equation
    \[
    \pi_{2*} ((a\Theta + b\kappa_{0,1} +\beta)^g \cup (t\kappa_{0,1} + \beta')) = 0,
    \]
    combined with Lemma \ref{lem:v}, \eqref{eq:dr1}, and the fact
    \[
    \pi_{2*}(\Theta^g \cup \beta') = g!\beta'
    \]
    given by the projection formula.
\end{proof}

\begin{cor}\label{cor:b}
    If $g\ge 2$, we have
    \begin{multline*}
        t \cdot \pi_{1*}\left(\frac{\Theta^{g-1}}{(g-1)!}\cup \frac{\kappa_{0,1}^3}{3!}\right)
    = t^2\cdot \pi_{2*}\left( \frac{\Theta^{g-1}}{(g-1)!}\cup \frac{\kappa_{0,1}^3}{3!}\right) \\
    + 4b\cdot \left( [\dr^2_g(b ;a_1)]_{b^4} - \frac{1}{2} ( [\dr^1_g(b;a_1)]_{b^2} )^2 \right).
    \end{multline*}
\end{cor}
\begin{proof}
    By (\ref{eq:int}) and Lemma \ref{lem:a}, we obtain
    \begin{equation}\label{eq:bv}
        \cc \cdot \pi_{1*}(\Theta^{g-1} \cup \kappa_{0,1}^3) = \pi_{2*} ((a\Theta + b\kappa_{0,1} +\beta)^{g-1} \cup (t\kappa_{0,1} + \beta')^3).
    \end{equation}
    Expanding the right-hand side, only the terms
    \begin{equation}\label{terms2}
        \pi_{2*}(\Theta^{g-1}\cup \kappa_{0,1}^3), \quad \pi_{2*}(\Theta^{g-1}\cup \kappa_{0,1}^2\cup \beta'), \quad \pi_{2*}(\Theta^{g-2}\cup \kappa_{0,1}^4)
    \end{equation}
    contribute by (\ref{ThetaKappa}) and Lemma \ref{lem:lowp}. 
    
    The second term of (\ref{terms2}) simplifies to
    \begin{align*}
        \pi_{2*}\big(\Theta^{g-1} \cup \kappa_{0,1}^2 \cup \beta') & = \pi_{2*}(\Theta^{g-1}\cup \kappa_{0,1}^2) \cup \beta'\\
        & = -2!(g-1)! [\dr^1_g(b;a_1)]_{b^2} \cup \frac{2b}{t}[\dr^1_g(b;a_1)]_{b^2}
    \end{align*}
    where the first equality follows from the projection formula and the second follows from \eqref{eq:dr1} and Lemma~\ref{lem:beta}.
    By Theorem \ref{thm:bmp}, the third term of (\ref{terms2}) equals to 
    \[\pi_{2*}(\Theta^{g-2}\cup \kappa_{0,1}^4) = 4!(g-2)! [\dr^2_g(b;a_1)]_{b^4}.\]
    We conclude the desired formula by comparing the coefficients of the terms in \eqref{eq:bv} and \eqref{terms2} respectively.
\end{proof}

\subsection{Explicit computation}\label{sec4.6}
Throughout this section we always assume that the stability condition $\phi$ is nondegenerate and semismall. To complete the proof of Theorem \ref{thm0.6}, we provide an explicit expression of the class
\begin{equation}\label{eq:explicit}
\pi_*(\Theta^{g-1}\cup \kappa_{0,1}^3)\in H^4(\Mbar_{g,1},\BQ), \quad \pi: \Jbar^{0,\phi}_{g,1} \to \Mbar_{g,1}
\end{equation}
in terms of tautological classes on $\Mbar_{g,1}$. This will be achieved in Proposition \ref{pro:kappa3} below which is the main result of this section.

The boundary stratum associated with a prestable graph containing unstable vertices corresponds to a lower-genus fine compactified Jacobian. Consider the boundary stratum $\Jbar^{0,\phi}_{\Gamma_\delta}$ of~$\Jbar^{0,\phi}_{g,1}$ corresponding to the prestable graph $\Gamma$ and the $\phi$-stable multidegree $\delta : V(\Gamma) \to \BZ$. Suppose that $\Gamma$ is obtained by subdividing $m$ edges of a stable graph $\Gamma_s$; we label these $m$~edges by $e_1, \dots, e_m$. Let $\widetilde{\Gamma}$ be the stable graph obtained from $\Gamma_s$ by separating the edges~$e_1,\dots,e_m$. We have a canonical identification $V(\Gamma_s) = V(\widetilde{\Gamma})$. 

For each edge $e_k = (h_k,h_k')$, we label the two additional legs of $\widetilde{\Gamma}$ by $2k$ and $2k+1$ (the order does not matter). By \cite[Lemma~2.4]{BMP}, there exists a nondegenerate stability condition~$\widetilde{\phi}$ of degree $-m$ over $\Mbar_{g-m,1+2m}$ such that the stratum $\Jbar^{-m,\widetilde{\phi}}_{\widetilde{\Gamma}_\delta}$ is identified by the following diagram
\begin{equation}\label{eq:stratum}
    \begin{tikzcd}
        \Jbar^{-m,\widetilde{\phi}}_{\widetilde{\Gamma}_\delta} \ar[d,"\pi_s"]\ar[r]\ar[rr, bend left, "\jmath"] & \Jbar^{-m,\widetilde{\phi}}_{g-m,1+2m}\ar[r]\ar[d] & \Jbar_{g,1}^{0,\phi} \ar[d,"\pi"]\\
        \Mbar_{\widetilde{\Gamma}}\ar[r,"j_{\widetilde{\Gamma}}"] &  \Mbar_{g-m,1+2m} \ar[r,"j_m"] & \Mbar_{g,1},
    \end{tikzcd}
\end{equation}
Here $j_m$ denotes the morphism gluing the $2k$-th and $(2k+1)$-th markings for $1 \le k \le m$, and $j_{\widetilde{\Gamma}}$ is the gluing morphism associated with $\widetilde{\Gamma}$. In particular, over $\Jbar^{-m,\widetilde{\phi}}_{\widetilde{\Gamma}_\delta}$, there are two universal curves: the first is a family of stable curves of genus $g-m$ with $1+2m$ markings, and the second is a family of quasistable curves of genus $g$ with one marking.

The genus $g-m$ universal curve over $\Jbar^{-m,\widetilde{\phi}}_{\widetilde{\Gamma}_\delta}$ together with \eqref{eq:pn} defines a morphism
\begin{equation}\label{eq:varphi2}
    \varphi_s : \Jbar^{-m,\widetilde{\phi}}_{\widetilde{\Gamma}_\delta} \to \prod_{v \in V(\widetilde{\Gamma})} \pic_{g(v),n(v),\delta(v)}.    
\end{equation}
For each vertex $v$, denote by 
\[
\Theta[v], \, \kappa_{0,1}[v], \,\xi_i[v], \,\psi_i[v] \in H^2(\Jbar^{-m,\widetilde{\phi}}_{\widetilde{\Gamma}_\delta}, \BQ) 
\]
the pullback of the corresponding classes from $\pic_{g(v),n(v),\delta(v)}$ along $\varphi_s$.

\begin{lem}\label{lem:lowg}
    Let $\jmath :\Jbar^{-m,\widetilde{\phi}}_{\widetilde{\Gamma}_\delta}\to \Jbar_{g,1}^{0,\phi}$ be the morphism in \eqref{eq:stratum}. 
    \begin{enumerate}
        \item[(i)] For $l\ge 1$, we have
        \[
        \Theta^l \cup [\Gamma_\delta] = \jmath_*\left(\sum_{v\in V(\widetilde{\Gamma})} \Theta[v] -\frac{1}{2}\sum_{k=1}^m (\xi_{2k}+\xi_{2k+1}) \right)^l.
        \]
        \item[(ii)] For $l \ge 1$, we have
        \[
        \kappa_{0,1}^l \cup [\Gamma_\delta] = \jmath_*\left(\sum_{v \in V(\widetilde{\Gamma})} \kappa_{0,1}[v]\right)^l.
        \]
        \item[(iii)] Let $e=(h,h')$ be an edge of $\Gamma$ connecting a stable vertex and an unstable vertex associated to $k$-th subdivided edge $e_k$, where $h'$ be the half-edge attached to the unstable vertex. Then we have
        \[
        [\Gamma_\delta,\psi_{h'}] = \jmath_*(\xi_{2k}-\xi_{2k+1}).
        \]
    \end{enumerate}
\end{lem}
\begin{proof}
    The genus $g$ universal quasistable curve over $\Jbar^{-m,\widetilde{\phi}}_{\widetilde{\Gamma}_\delta}$ induces a morphism
    \[
    \varphi : \Jbar^{-m,\widetilde{\phi}}_{\widetilde{\Gamma}_\delta} \to \prod_{w \in V(\Gamma)} \pic_{g(w),n(w),\delta(w)}.
    \]
    For each $w\in V(\Gamma)$, we have the natural projection
    \[
    p_w :\prod_{w\in V(\Gamma)}\pic_{g(w),n(w),\delta(w)}\to \pic_{g(w),n(w),\delta(w)}.
    \]

    Consider the following relation in $\Chow^*(\pic_{g,n,d})$:
    \[
    \Theta \cup [\Gamma_\delta] = \sum_{w \in V(\Gamma)} [\Gamma_\delta,\varphi^*p_w^*\Theta]. 
    \]
    When $w \in V(\Gamma)$ is a stable vertex, then $\varphi^*p_w^* \Theta = \Theta[w]$.
    Therefore, it is enough to consider the case when $w$ is an unstable vertex of $\Gamma$. When $w$ is an unstable vertex, the degree at $w$ should be $1$ by the stability condition. The relation $\uniDR^1_0(0;a+1,-a)=0$ on $\pic_{0,2,1}$ yields
    \[
    \Theta + \frac{1}{2}(\xi_1+\xi_2) + \text{(other terms)} = 0.
    \]
    After pulling back this relation along $p_w\circ \varphi$, the \emph{other terms} do not contribute due to the stability conditions. Therefore, the result holds for $m=1$. The general case of (i) follows by a similar argument.

   The argument for (ii) is parallel. It suffices to consider $\kappa_{0,1}[w]$ with $w$ an unstable vertex of $\Gamma$, and then it follows from the relation $\uniDR^1_0(b;a, -a) = 0$ on $\pic_{0,2,1}$.

   Part (iii) follows from the relation $\uniDR^1_0(a+1,-a)=0$ on $\pic_{0,2,-1}$; see \cite[(34)]{BMP}.
\end{proof}

We return to the computation of the pushforward (\ref{eq:explicit}); the contributions from strata away from the tree-like locus will play an important role.

We consider the classes
\[
[\Gamma_1], \ldots, [\Gamma_{g-1}] \in H^4(\Mbar_{g,1}, \BQ) 
\]
given by the stable graphs
\[
\Gamma_{g_1}:=\Big[\begin{tikzpicture}[baseline=-0.1 cm, vertex/.style={circle,draw,font=\Large,scale=0.5, thick}]
    \node[vertex] (A) at (-1.5,0) [label=below:$g_2$, label=above:$v_{g_1}$] {};
    \node[vertex] (B) at (0,0) [label=below:$g_1$, label=above:$w_{g_1}$] {};
    \draw[thick] (A) to[bend left] (B);
    \draw[thick] (A) to[bend right] (B);
    \draw[thick] (A) to (-2, 0) node[left] {$1$};
    \end{tikzpicture}\Big],
\]
where $g_1$ and $g_2$ denote the genera of the two vertices $w_{g_1}$ and $v_{g_1}$ with $g_1+g_2 = g-1$ and~$g_1 \geq 1$. 
    
By \cite[Corollary~3.6]{KP}, the stability condition $\phi$ over $\Mbar_{g,1}$ is uniquely determined by its values at stable graphs of the~form 
    \begin{enumerate}
        \item[(i)] two vertices with one edge connecting them, and
        \item[(ii)] $\Gamma_{g-1}$.
    \end{enumerate}    
For an integer $z$, let $\phi(z)$ be a sequence of nondegenerate semismall stability conditions which are identical for all the graphs in (i) above, and for $\Gamma_{g-1}$ we assign  
\[
\phi(z)(v_{g-1}) = z + \epsilon, \quad \phi(z)(w_{g-1}) = -z - \epsilon, \quad 0<\epsilon\ll \frac{1}{g-1}.
\]
The values of $\phi(z)$ on the other graphs $\Gamma_{g_1}$ are determined by the following formula.
\begin{lem}\label{lem:phi}
    Let $\phi(z)$ be the nondegenerate semismall stability condition as above. Then the value of $\phi(z)$ on the graph $\Gamma_{g_1}$ is given by
    \[
    \phi(z)(v_{g_1}) = \frac{g_1}{(g-1)} \cdot (z + \epsilon),\quad  \phi(z)(w_{g_1}) = -\frac{g_1}{(g-1)} \cdot (z + \epsilon).    \]
\end{lem}
\begin{proof}
    Consider the stable graph $\widetilde{\Gamma}$ with $g$ vertices labeled $u_0, u_1, \dots ,u_{g-1}$, where $g(u_0)=0$ and $g(u_1)=\cdots=g(u_{g-1})=1$. For $i=0, \dots, g$, the vertex $u_i$ is connected to $u_{i-1}$ by one edge (with $u_g := u_0$), and the first marking is attached to $u_0$.

   We show that the values $\phi(z)(u_1), \dots, \phi(z)(u_{g-1})$ are all equal. For each $1 \le i \le g-2$, consider the contraction of graphs $\widetilde{\Gamma} \to \widetilde{\Gamma}_i$ in which the vertices $u_k$ for $k \ne i,i+1$ are contracted to $u_0$. The graph $\widetilde{\Gamma}_i$ admits a symmetry interchanging the vertices $u_i$ and $u_{i+1}$. Since $\phi(z)$ is compatible with both contractions and the symmetries of stable graphs, we obtain
    \[
    \phi(z) (u_i) = \phi(z) (u_{i+1}), \quad 1\le i\le g-2.
    \]
    Therefore we have $\phi(z) (u_1) = \cdots = \phi(z) (u_{g-1})$.
    
    On the other hand, there is a contraction $\widetilde{\Gamma} \to \Gamma_0$ contracting the vertices $u_i$ for $1 \le i \le g-1$ to a single vertex. Compatibility of $\phi(z)$ with contractions then yields
    \[
    \phi(z) (u_0) = \phi(z) (v_{g-1}).
    \]
    Since $\phi(z)$ is a stability condition of degree $0$, we have
    \[
    \phi(z)(u_0) + (g-1)\phi(z) (u_1) = 0,
    \]
    and the desired equality follows from the compatibility of $\phi(z)$ with contractions.
\end{proof}    

Let
\[
\delta: V(\Gamma) \to \BZ
\]
be a multidegree which is stable with respect to $\phi(z)$. On the graph $\Gamma_{g_1}$, the stability condition (see \cite[Definition~4.1]{KP}) forces
\begin{equation*}\label{eq:st}
\big|\delta(v_{g_1}) - \phi(z)(v_{g_1})\big| \le 1.
\end{equation*}
So there are only two possible $\phi(z)$-stable multidegrees, and we denote them by 
\begin{equation}\label{no1_stable}
\delta(v_{g_1}) = \delta_{g_1}(z), \quad \delta(w_{g_1}) = -\delta_{g_1}(z)
\end{equation}
and 
\begin{equation}\label{no2_stable}
\delta(v_{g_1}) = \delta_{g_1}(z)+1, \quad \delta(w_{g_1}) = - \delta_{g_1}(z)-1
\end{equation}
respectively. In particular, the constant $\delta_{g_1}(z)$ is an invariant dependent on the stability condition $\phi(z)$.

Let $\Gamma_{g_1}'$ be the quasistable graph obtained by subdividing one of the edges of $\Gamma_{g_1}$. Denote by $v_{g_1}'$ and $w_{g_1}'$ the associated stable vertices of $\Gamma_{g_1}'$. The stability condition  forces
\begin{equation}\label{eq:st2}
\delta(v_{g_1}') = \delta_{g_1}(z), \quad \delta(w_{g_1}') = -\delta_{g_1}(z)-1.
\end{equation}
In particular, there is a unique $\phi(z)$-stable multidegree on $\Gamma_{g_1}'$.

For the fibration 
\[
\pi: \Jbar^{0,\phi(z)}_{g,1} \to \Mbar_{g,1}, 
\]
we can express the pullback of the stratum class $[\Gamma_{g_1}]$ as
\[
\pi^*[\Gamma_{g_1}] = [\Gamma_{g_1}^+] + [\Gamma_{g_1}^-],
\]
where the classes $[\Gamma_{g_1}^+]$ and $[\Gamma_{g_1}^-]$ correspond to the strata with multidegrees (\ref{no1_stable}) and (\ref{no2_stable}) respectively.

\begin{lem}\label{lem:2g1}
Assume that $1 \leq g_1 \leq g-1$.
\begin{enumerate}
    \item[(i)] We have the relation
     \[
    \pi_{*}\left( \frac{\Theta^{g-1}\cup \kappa_{0,1}}{(g-1)!} \cup [\Gamma_{g_1}^+]
    \right) = 2g_1 \cdot [\Gamma_{g_1}].
    \]
    \item[(ii)] We have the vanishing
     \[
    \pi_*( \Theta^{g-2} \cup\kappa_{0,1}^2 \cup [\Gamma^+_{g_1}]) = \pi_*( \Theta^{g-2} \cup\kappa_{0,1}^2 \cup [\Gamma^-_{g_1}])= 0.
    \]
    \end{enumerate}
   
\end{lem}
\begin{proof}[Sketch of the proof]
The universal double ramification cycle relations provide an algorithm to calculate the classes of the form
\begin{equation}\label{pushforward1}
\pi_{*}( \Theta^{k} \cup \kappa^l_{0,1} \cup [\Gamma_{g_1}^\pm]) \in H^4(\Mbar_{g,1}, \BQ), \quad k+l = g,
\end{equation}
which leads to the proof of both identities above. Since the full proof is lengthy but the idea is straightforward, we summarize the computational scheme here and leave the details to Section~\ref{sec:4.9}.

A class (\ref{pushforward1}) must be a multiple of the class $[\Gamma_{g_1}]$, and it suffices to determine the coefficient. First, we use Lemma \ref{lem:lowg} to distribute the monomial $\Theta^{k}\cup \kappa^l_{0,1}$ among the vertices of $\Gamma_{g_1}^\pm$; then we apply the relations in Section \ref{sec:4.3} at each vertex. For dimension reasons, the only terms that can contribute nontrivially to the pushforward are those corresponding to the stratum~$\Gamma_{g_1}'$, decorated by monomials of $\Theta$ divisors on stable vertices. These terms can be computed by successive applications of the universal double ramification cycle relations.
\end{proof}

\begin{prop}\label{pro:kappa3}
    We have
    \[
    \pi_{*}\left(\frac{\Theta^{g-1}}{(g-1)!} \cup \frac{\kappa_{0,1}^3}{3!} \right) = \sum_{g_1 = 1}^{g-1} a_{g_1}(\phi(z)) \cdot [\Gamma_{g_1}]
    \]
    where the constants $a_{g_1}(\phi(z)) \in \BQ$ depend on the stability condition $\phi(z)$ as
    \[
    a_{g_1}(\phi(z)) = - \frac{4(2g_1)^3}{48}\left(\delta_{g_1}(z) - \frac{1}{2}\right).
    \]
\end{prop}
\begin{proof}
    We multiply $\tfrac{1}{3}\kappa_{0,1}$ on each side of \eqref{eq:rel22}. By \eqref{eq:kappa1} and the projection formula, the first term vanishes:
    \[
    \pi_{*}(\Theta^{g} \cup \kappa_{0,1} \cup \widetilde{\kappa}_1) = \pi_{*}(\Theta^g\cup \kappa_{0,1}) \cup \widetilde{\kappa}_1 = 0. 
    \]

    The third and the fourth terms also do not contribute to the pushforward. To see this, by the vanishing Lemma~\ref{lem:2g1}(ii), we have for any $g_1 \ge 1$ the vanishing
    \[
    \pi_*( \Theta^{g-2} \cup\kappa_{0,1}^2 \cup [\Gamma_{g_1}^+]) = \pi_*( \Theta^{g-2} \cup\kappa_{0,1}^2 \cup [\Gamma_{g_1}^-]) = 0.
    \]
    When $g_1 = 0$, the corresponding term vanishes because the coefficient is a multiple of $g_1$. A similar argument shows that the fourth term also does not contribute to the pushforward.
    
    Thus only the second term contributes to the pushforward. By Lemma~\ref{lem:2g1} and \eqref{no1_stable}, the contribution from the stratum $\Gamma_{g_1}^+$ pushes forward to
    \[
    \frac{-2(2g_1)^3}{48}(\delta_{g_1}(z))^2 \cdot [\Gamma_{g_1}].
    \]
    Since $\pi^*[\Gamma_{g_1}] = [\Gamma_{g_1}^+] + [\Gamma_{g_1}^-]$, the projection formula implies
    \[
    \pi_{*}\left( \frac{\Theta^{g-1}\cup \kappa_{0,1}}{(g-1)!} \cup [\Gamma_{g_1}^-]
    \right) = - 2g_1 \cdot [\Gamma_{g_1}].
    \]
    Therefore, by \eqref{no2_stable}, the contribution from the stratum $\Gamma_{g_1}^-$ is
    \[
    \frac{2(2g_1)^3}{48}((\delta_{g_1}(z) +1)^2)\cdot [\Gamma_{g_1}].
    \]
    Adding the two contributions yields the formula for $a_{g_1}(\phi(z))$.
\end{proof}

\begin{lem}\label{lem:Gamma}
    The classes $[\Gamma_1], \ldots, [\Gamma_{g-1}]$ are linearly independent in $H^4(\Mbar_{g,1},\BQ)$.
\end{lem}
\begin{proof}
    We proceed by explicit intersection theory. When $g_1\neq i$, the vanishing 
    \[
    \int_{\Gamma_{g_1}} \psi_1^{3g-3-3i}\cup \kappa_{3i-1} =0
    \]
    holds for dimension reasons. On the other hand, we have
    \[\int_{\Gamma_i}\psi_1^{3g-3-3i}\cup \kappa_{3i-1} =\int_{\Mbar_{g-1-i,3}}\psi_1^{3g-3-3i}\cdot\int_{\Mbar_{i,2}}\kappa_{3i-1} = \int_{\Mbar_{g-1-i}}\kappa_{3g-3i-6}\cdot\int_{\Mbar_{i}}\kappa_{3i-3}\]
    which is nonzero by \cite[(28)]{FP}.
\end{proof}

\subsection{Proof of Theorem \ref{thm0.6}}
    Suppose that there is an isomorphism
    \[
    f: H^*(\Jbar^{0,\phi(z)}_{g,1}, \BQ) \rightarrow H^*(\Jbar^{0,\phi(z')}_{g,1}, \BQ)
    \]
of $H^*(\overline{\CM}_{g,1}, \BQ)$-algebras sending
\[
f(\Theta) = a\Theta + b\kappa_{0,1} + \beta, \quad f(\kappa_{0,1}) = s \Theta + t \kappa_{0,1} + \beta'.
\]
By Lemmas \ref{lem:a}, \ref{lem:v}, and \ref{lem:beta}, we must have
\[
s = 0, \quad a = t^2, \quad \beta' = \frac{2b}{t} [\dr^1_g(b;a_1)]_{b^2}.
\]
By an explicit calculation, we get
\begin{equation*}
        [\dr^2_g(b ; a_1)]_{b^4} - \frac{1}{2} ( [\dr^1_g(b; a_1)]_{b^2} )^2 =  \sum_{g_1=1}^{g-1} \frac{(2g_1)^4}{48} \cdot [\Gamma_{g_1}].
    \end{equation*}
Therefore, by Proposition \ref{pro:kappa3} and Lemma \ref{lem:Gamma}, both sides of the equation in Corollary \ref{cor:b} can uniquely be expressed in terms of linear combinations of $\Gamma_{g_1}$ with $g
_1=1,\ldots, g-1$; comparing their coefficients, we obtain
\begin{equation}\label{eq:aphi}
    t \left(\delta_{g_1}(z) - \frac{1}{2}\right) = t^2\left(\delta_{g_1}(z')- \frac{1}{2}\right) - 4b(2g_1),\quad g_1=1,\ldots, g-1.
\end{equation}

To complete the proof, it suffices to find two degrees $z,z'$ such that \eqref{eq:aphi} has no solutions with $b \in \BQ, t \in \BQ^\times$.

In fact, we just take:
\[
z= 0, \quad  z'=2.
\]
When $z=0$, we have $\delta_{g_1}(0) = 0$ for all $g_1$. On the other hand, when $z'=2$, by \eqref{eq:st} and Lemma \ref{lem:phi} we obtain
\[
\delta_{g_1}(2) = \begin{cases}
    0 & \text{if } g_1<\frac{g-1}{2}\\
    1 & \text{if } g_1\ge \frac{g-1}{2}.
\end{cases}
\]
For $g\ge 4$, a nontrivial rational solution $(b,t)$ to \eqref{eq:aphi} must satisfy
\[
-\frac{t}{2}=-\frac{t^2}{2} -8b, \quad -\frac{t}{2}=\frac{t^2}{2} -8b(g-1), \quad -\frac{t}{2}=\frac{t^2}{2} -8b(g-2)
\]
by considering $g_1=1,g-2,g-1$. This is clearly impossible which completes the proof of Theorem \ref{thm0.6}.
\qed

\begin{rmk}
    When $t = 1$, the condition \eqref{eq:aphi} is always satisfied for a pair of integers $z,z'$~with 
    \[
    z' = z + b(2g-2), \quad b\in \BZ. 
    \] 
    This is because of the existence of an isomorphism
    \[\Jbar_{g,1}^{0,\phi(z)}\xrightarrow{\simeq} \Jbar_{g,1}^{0,\phi(z')}, \quad (C,x_1,L) \mapsto (C, x_1, L \otimes \omega_{C,\log}^{\otimes -b}\otimes \CO_C(b(2g-1)x_1)\otimes\CO(\alpha))\]
    inducing the obvious ring isomorphism $f$. Here $\alpha$ is the unique vertical divisor on the universal curve $\mathcal{C}_{g,1} \to \Mbar_{g,1}$ chosen so that $L \otimes \omega_{C,\log}^{\otimes -b}\otimes \CO_C(b(2g-1)x_1)\otimes\CO(\alpha)$ has multidegree $0$ over~$\Mtl_{g,1}$.
\end{rmk}

\subsection{Proof of Lemma \ref{lem:2g1}}\label{sec:4.9}
Now we prove Lemma \ref{lem:2g1}. We provide here the details for the proof of (i); the proof of (ii) is similar and simpler.

   For notational convenience, we set 
    \[
    v := v_{g_1}, \quad w := w_{g_1}, 
    \]
    so that $g(w) = g_1$:
    \begin{equation*}
\Gamma_{g_1}=\Big[\begin{tikzpicture}[baseline=-0.1 cm, vertex/.style={circle,draw,font=\Large,scale=0.5, thick}]
    \node[vertex] (A) at (-1.5,0) [label=below:$g_2$, label=above:$v$] {};
    \node[vertex] (B) at (0,0) [label=below:$g_1$, label=above:$w$] {};
    \draw[thick] (A) to[bend left] (B);
    \draw[thick] (A) to[bend right] (B);
    \draw[thick] (A) to (-2, 0) node[left] {$1$};
    \end{tikzpicture}\Big].
\end{equation*}    
    Let $e = (h,h')$ be one of the two edges of $\Gamma_{g_1}$ with $h$ the half-edge attached to $w$. We consider the morphism \eqref{eq:varphi2}. By Lemma \ref{lem:lowg}(i, ii), we have
    \begin{multline}\label{eq:2g}
        \frac{\Theta^{g-1}\cup \kappa_{0,1}}{(g-1)!}\cup [\Gamma^+_{g_1}] = \sum_{j_1 +j_2 = g-1} \left[ \Gamma_{g_1}^+, \, \frac{\kappa_{0,1}[v]\cup (\Theta[v])^{j_1} \cup (\Theta[w])^{j_2}}{j_1! j_2!} \right] \\
        + \sum_{k_1 + k_2 = g-1} \left[\Gamma_{g_1}^+,\, \frac{\kappa_{0,1}[w]\cup (\Theta[v])^{k_1} \cup (\Theta[w])^{k_2}}{k_1! k_2!} \right].
    \end{multline}

\medskip
{\bf \noindent (i) The first summation.}  We show that the first summation of \eqref{eq:2g} does not contribute to the pushforward along $\pi: \Jbar^{0,\phi(z)}_{g,1} \to \Mbar_{g,1}$.
    
\medskip
    {\bf \noindent Case 1: $j_1\geq g(v)$.}
    We apply the relation \eqref{eq:rel3} at the vertex $v$, with the second marking of $\pic_{g(v),3,\delta(v)}$ identified with the first marking of $\Gamma_{g_1}^+$. Since the universal line bundle is trivialized along the first marking, we have $\xi_2[v] = 0$. In \eqref{eq:rel3}, the terms supported on unstable boundary strata do not contribute, because the vertex $v$ cannot further degenerate to a stratum containing an unstable vertex (recall that any unstable vertex must have degree~$1$) by the stability condition \eqref{eq:st2}. The terms supported on stable boundary strata also do not contribute for dimension reasons. Therefore, only the second term of \eqref{eq:rel3} remains. 
    
    For the monomials $(\Theta[v])^{m} \cup (\psi_2[v])^n$ with $m \ge g(v)+1$, we can apply \eqref{eq:rel5} at the vertex~$v$ repeatedly to reduce to monomials with $m \le g(v)$ and $n \ge 1$. However, since $\psi_2[v] = \pi^*\psi_1$, we obtain
    \[
    \pi_*((\psi_2[v])^n\cup (\Theta[v])^m\cup (\Theta[w])^{j_2}) = \psi_1^n \cup \pi_*(\Theta[v])^m\cup (\Theta[w])^{j_2}) = 0
    \]
    where the last equality holds for dimension reasons. Therefore, monomials with $j_1 \ge g(v)$ do not contribute to the pushforward.
    
\medskip
    {\bf \noindent Case 2: $j_1 < g(v)$.}
    In this case we have $j_2> g(w)$. Hence the relation \eqref{eq:rel5} can be applied at the vertex $w$. Let $\mathrm{st} : \mathfrak{M}_{g(w),2} \to \Mbar_{g(w),2}$ be the stabilization morphism. By \cite[Proposition~3.14]{BS1}, we have
    \begin{equation}\label{eq:psi}
        \mathrm{st}^*\psi_1 = \psi_1 - \Big[\begin{tikzpicture}[baseline=-0.1 cm, vertex/.style={circle,draw,font=\Large,scale=0.5, thick}]
    \node[vertex] (A) at (-1.5,0) [label=below:$0$] {};
    \node[vertex] (B) at (0,0) [label=below:$g(w)$] {};
    \draw[thick] (A) to (B);
    \draw[thick] (A) to (-2, 0) node[left] {$1$};
    \draw[thick] (B) to (0.5,0) node[right]{$2$};
    \end{tikzpicture}\Big].
    \end{equation} 
    
    First, consider the terms without boundary strata. After repeatedly applying the relation~\eqref{eq:rel5}, such terms reduce to a linear combination of
    \[(\Theta[w])^{g(w)} \cup (\xi_2[w])^{j_3}\cup (\mathrm{st}^*\psi_2[w])^{j_4}, \quad j_3+j_4\geq 1.\]
    If $j_4 \geq 1$, then the pushforward of
    \[ \left[\Gamma_{g_1}^+,\, \kappa_{0,1}[v]\cup (\Theta[v])^{j_1}\cup (\Theta[w])^{g(w)} \cup (\xi_2[w])^{j_3}\cup (\mathrm{st}^*\psi_2[w])^{j_4} \right] \]
    is zero for dimension reasons. Hence, we assume $j_4=0$ and $j_3 \geq 1$.
    We identify the half-edge~$h$ with the second marking of $w$, and $h'$ with the first marking of $w$. Since
    \[
    \xi_2[w] = \xi_{h}=\xi_{h'} = \xi_1[v],
    \]
    we may move $(\xi_2[w])^{j_3}$ to the vertex $v$. Then the codimension of the tautological class at~$v$ exceeds $g(v)$; by a similar argument as above, this term does not contribute after the pushforward. 
    
    Next, consider the terms with boundary strata. For dimension reasons, the only possible contributions are divisor classes with an unstable vertex. If the vertex $w$ is decorated by~$\mathrm{st}^*\psi_2[w]$, the pushforward again vanishes for dimension reasons. By Lemma~\ref{lem:lowg}, we are reduced to computing the pushforward of
    \[
    \jmath_*(\kappa_{0,1}[v] \cup (\Theta[v])^{j_1}\cup (\xi_1[v])^{j_2-g(w)-1} \cup (\Theta[w])^{g(w)})
    \]
    where $\jmath : \Jbar_{g(v),3}^{\delta(v), \phi_v} \times \Jbar^{-\delta(v)-1,\phi_w}_{g(w),2} \to \Jbar_{g,1}^{0,\phi}$ is defined in \eqref{eq:stratum}. The cohomology class from the first factor has perversity
    \[2j_1+j_2-g(w)= 2g(v) +g(w) - j_2 < 2g(v).\]
    Therefore, by Lemma~\ref{lem:lowp}, the pushforward of this term also vanishes. 

\medskip
{\bf \noindent (ii) The second summation.}    We return to \eqref{eq:2g}. By a similar argument as above, in the second summation on the right-hand side, only the term
    \[
    \left[\Gamma_{g_1}^+,\, \frac{(\Theta[v])^{g(v)}}{g(v)!} \cup \frac{(\Theta[w])^{g(w)}\cup \kappa_{0,1}[w]}{g(w)!} \right]
    \]
    contributes to the pushforward along $\pi$.
    By \eqref{eq:rel3} on $\pic_{g(w),2,-\delta(v)}$, we have
    \begin{multline}\label{eq:cst1}
        \frac{(\Theta[w])^{g(w)}\cup \kappa_{0,1}[w]}{g(w)!} = 2g(w)\cdot \frac{(\Theta[w])^{g(w)}\cup \xi_h}{g(w)!} - 2g(w)\delta(v) \cdot \frac{(\Theta[w])^{g(w)}\cup \psi_h}{g(w)!} \\
        + \frac{(\Theta[w])^{g(w)}}{g(w)!} \cup \sum_{G_1+G_2 = g(w)}\sum_{e+e' = -\delta(v)} -2G_1e \cdot \Big[\begin{tikzpicture}[baseline=-0.1 cm, vertex/.style={circle,draw,font=\Large,scale=0.5, thick}]
        \node[vertex] (A) at (-1.5,0) [label=below:$G_1$] {$e$};
        \node[vertex] (B) at (0,0) [label=below:$G_2$] {$e'$};
        \draw[thick] (A) to (B);
        \draw[thick] (A) to (-2, 0) node[left] {$h'$};
        \draw[thick] (B) to (0.5,0) node[right]{$h$};
        \end{tikzpicture}\Big].   
    \end{multline}
    In the following, we compute the contribution of each individual term on the right-hand side of \eqref{eq:cst1}.
    
    For the first term of \eqref{eq:cst1}, consider the relation \eqref{eq:rel4} on $\pic_{g(v),3,\delta(v)}$, with the second marking identified by $h'$. Since we have trivialized the universal line bundle along the first marking, it follows that $\xi_1=0$. Therefore,
    \[
    \pi_*\left( (\Theta[v])^{g(v)}\cup (\Theta[w])^{g(w)}\cup \xi_h\right) = \pi_*\left( (\Theta[v])^{g(v)}\cup \xi_{h'} \cup (\Theta[w])^{g(w)}\right) = 0,
    \]
    where the first equality uses $\xi_h = \xi_{h'}$, and the second follows from the fact that the boundary stratum in \eqref{eq:rel4} does not contribute under the stability condition \eqref{eq:st2}. This shows that the first term of \eqref{eq:cst1} does not contribute to the pushforward.

    For the second term of \eqref{eq:cst1}, we have
    \[
    \begin{aligned}
        \pi_*\left[\Gamma_{g_1}^+,\, \frac{(\Theta[v])^{g(v)}\cup (\Theta[w])^{g(w)}\cup \psi_h}{g(v)! g(w)!} \right]  &= \pi_*\left[\Gamma_{g_1}',\, \frac{(\Theta[v])^{g(v)}\cup (\Theta[w])^{g(w)}}{g(v)! g(w)!} \right] \\
        &= [\Gamma_{g_1}],
    \end{aligned}
    \]
    where the first equality follows from \eqref{eq:psi} together with the projection formula, and the second from Lemma~\ref{lem:lowg}. Hence the second term of \eqref{eq:cst1} contributes $-2g(w)\delta(v) \cdot [\Gamma_{g_1}]$. 

    For the third term of \eqref{eq:cst1}, the boundary stratum contributes to the pushforward only when it is unstable for dimension reasons; hence $G_1 = 0$ or $G_1 = g(w)$. Since the coefficient is proportional to $G_1$, we only consider the case $G_1 = g(w)$; hence $e = -\delta(v)-1$ by \eqref{eq:st2}. Therefore, the third term of \eqref{eq:cst1} contributes $2g(w)(\delta(v)+1)\cdot [\Gamma_{g_1}]$.

    After summing all contributions of \eqref{eq:cst1}, we obtain the desired formula.\qed

\appendix

\section{Motives of torsors under abelian schemes}


We consider the relative Chow motive with $\mathbb{Q}$-coefficients of a smooth abelian fibration $f: T\to B$ over a nonsingular base variety $B$; that is, $f$ is proper and its geometric fibers are abelian varieties. The degree $d$ relative Jacobian $J_C^d \to B$ associated with a family of nonsingular curves $C\to B$ is an example. Generally, any smooth abelian fibration $f: T\to B$ is a torsor under an abelian scheme $g: A \to B$; see \cite[Proposition 2.1]{JM}. We prove in the following that, from the motivic perspective, a torsor behaves exactly as an abelian scheme. The argument was obtained in correspondence with Dan Petersen and we thank him for these discussions.

\begin{thm}\label{thm0}
Assume that $B$ is a nonsingular variety over a field. Let $f: T \to B$ be a torsor under an abelian scheme $g: A \to B$. Then there is an isomophism of relative Chow motives
    \[
    h(T/B) \simeq h(A/B)
    \]
which respects the cup-product. In particular, the relative Chow motive $h(T/B)$ admits a multiplicative Chow--K\"unneth decomposition.
\end{thm}

\begin{proof}
By a general result of Raynaud \cite[Corollary XIII 2.4 and Proposition XIII 2.6]{Ray}, for any regular and quasicompact base $B$ and any abelian scheme $A \to B$, the first \'etale cohomology~$H^1_{\text{\'et}}(B, A)$ is a torsion group. This implies the existence of a positive integer $N$ for which the pushforward torsor
\[
[N]_*T := T \overset{A}{\times} A
\]
via the ``multiplication by $N$'' map $[N]: A \to A$ is isomorphic to the trivial torsor $A$. Here the $A$-action on the second factor $A$ is defined through $[N]$.

The natural morphism $T \to [N]_*T$ then yields a morphism
\[
\varphi: T \to A
\]
which is fiberwise an isogeny up to translation over $B$.

We consider the morphism of relative Chow motives 
\begin{equation}\label{motives}
\varphi^*: h(A/B) \rightarrow h(T/B).
\end{equation}
Since (\ref{motives}) is induced by the pullback of a morphism, it automatically preserves the cup-product, and it is an isomorphism by the conservativity result \cite[Lemma A.6]{AHPL}.

More precisely, we observe that it is enough to prove the isomorphism after base change to the algebraic closure of the base field. When the base field is algebraically closed (or more generally a perfect field), the category of relative Chow motives over~$B$, in the sense of either~\cite{DM} or \cite{CH}, admits a fully faithful embedding in the category of Beilinson motives $\mathrm{DM}_{\RB}(B)$ of~\cite{CD}; see \emph{e.g.}~\cite{Jin}. By the comparison theorem \cite[Theorem~16.2.18]{CD} between $\mathrm{DM}_{\RB}(B)$ and the category of \'etale motives $\mathrm{DA}_{\text{\'et}}(B, \BQ)$ with $\mathbb{Q}$-coefficients (and without transfers), and by \cite[Lemma~A.6]{AHPL}, it suffices to verify that \eqref{motives} is an isomorphism when pulled back to any geometric point of $B$. The latter statement is immediate since $\varphi$ is fiberwise an isogeny up to translation.

Finally, the last claim follows from the construction of the multiplicative Chow--K\"unneth decomposition for abelian schemes in \cite{DM}.
\end{proof}

\end{document}